\documentclass[10pt]{article}
\usepackage{amsmath, amsfonts, amsthm}
\usepackage{amssymb,mathrsfs, verbatim}
\usepackage{graphicx,amstext}
\usepackage{amsmath}
\usepackage{mathtools}
\usepackage{amssymb}
\usepackage{latexsym}
\usepackage{ mathrsfs }
\usepackage{enumerate}
\usepackage[USenglish]{babel}

\usepackage[latin1]{inputenc}
\newcommand{\clg}[1]{{\mathcal{#1}}}

\newcommand{\ol}[1]{{\overline{#1}}}

\newcommand{\R}{\mathbb R}

\newcommand{\N}{\mathbb N}

\newcommand{\vp}{\varphi}

\newcommand{\ess}{\,{\rm ess}}

\newcommand{\dive}{{\rm div}}

\numberwithin{equation}{section}

\newtheorem{theorem}{Theorem}[section]
\newtheorem{proposition}{Proposition}[section]
\newtheorem{remark}{Remark}[section]
\newtheorem{lemma}{Lemma}[section]
\newtheorem{corollary}{Corollary}[section]
\newtheorem{definition}{Definition}[section]

\begin{document}
\title{INITIAL BOUNDARY VALUE PROBLEM FOR ANISOTROPIC
FRACTIONAL TYPE DEGENERATE PARABOLIC EQUATION}
\author{Gerardo Huaroto}
%\address{\emph{Instituto de Matem\'{a}tica, Universidade Federal de Alagoas\\ Macei\'o AL-Brazil}}
%\email{Gerardo@.....}
\date{}

\maketitle

\footnotetext[1]{ Universidade Federal
de Alagoas. E-mail: {\sl gerardo.cardenas@im.ufal.br}.
}

%%%%%%%%%%%%%%%%%%%%%%%%%%%%%%%%%%%%%%%%%%%%%%%%%%%%%%%%%%%%%%%%%%%%%%%%%%%%
\begin{abstract}
The aim of the paper is to generalize the author's previous work \cite{Hua}. We extend the argument \cite{Hua}  for any uniformly elliptic operator in divergence form $\mathcal{L}u=-\dive(A(x)\nabla u)$, more precisely, we study  a fractional type degenerate elliptic equation posed in bounded domains with homogeneous boundary conditions
$$
\partial_t u=\dive(u\,A(x)\nabla \mathcal{L}^{-s}u)
$$ 
where $\mathcal{L}^{-s}$ is the inverse $s$-fractional elliptic operator for any $s \in (0,1)$. This work consists of two part. The first part is devoted to state how the boundary condition will be consider (in the spirit of F. Otto \cite{Otto}), and to give a formulation for the IBVP. In the second part, It is shown the existence of mass-preserving, non-negative weak solutions satisfying energy estimates for measurable and bounded non-negative initial data. 
\end{abstract}
%%%%%%%%%%%%%%%%%%%%%%%%%%%%%%%%%%%%%%%%%%%%%%%%%%%%%%%%%%%%%%%%%%%%%%%%%%%%

\maketitle

%%%%%%%%%%%%%%%%%%%%%%%%%
\section{Introduction}
\label{INTRO}
%%%%%%%%%%%%%%%%%%%%%%%%%
The aim of this paper is study the existence of solution of \eqref{FTPME}. More precisely, 
to state how the boundary conditions will be consider, and to express in a convenient way the concept of solution for the following problem
% in what sense the solution of \eqref{FTPME} is understood,
%  and to show the existence of solutions of the following problem
\begin{equation}
\label{FTPME} 
     \left \{
       \begin{aligned}
       \partial_t u&= \dive (u\, A(x) \, \nabla \clg{K} u) \quad \text{in $\Omega_T$},
       \\
       u|_{t= 0} &= u_0 \hspace{2.6cm}\quad \text{in $\Omega$},
       \\
       u &= 0 \hspace{2,8cm}\quad \text{on $(0,T) \times \partial \Omega$},
\end{aligned}
\right.
\end{equation}
where $\Omega_T:= (0,T) \times \Omega$, for any real number $T> 0$, and
$\Omega \subset \R^n$ $(n\geq1)$ is a bounded open set having smooth ($C^2$)
boundary $\partial \Omega$.
%Here, $u(t,x)$ is an unknown real function, which can be physical, an absolute temperature, or a density, also a concentration, thus non-negative. 
%
Moreover, the initial data $u_0$ is a measurable, bounded 
non-negative function in $\Omega$, and considered homogeneous 
Dirichlet boundary condition, while $\clg{K}:=\mathcal{L}^{-s}$, is the inverse of the $s$-fractional elliptic operator  (see Definition \ref{DSFE}), and the matrix $A(x)=(a_{ij}(x))_{n\times n}$ satisfy the uniform ellipticity condition. 

\smallskip
The nonlocal, possibly degenerate, parabolic type equation is inspired in a non-local Fourier's law, that is
$$
   \mathbf{q}:= -\kappa(x,u) \ \nabla \clg{K} u,
$$
where $u$ is the temperature, $\mathbf{q}$ is the diffusive flux, and $\kappa(x,u)$ denotes here
the (non-negative definite) thermal conductivity tensor. 
%For instance, if we suppose  $\kappa(x,u)=A(x)$, we obtain
%$$
%   \partial_t u=\dive(A(x)\nabla \mathcal{K}u),
%$$
%which is the fractional version (written in the divergence form) 
%of the standard parabolic equation. 

\smallskip
Equation \eqref{FTPME} is motivated  in the so-called Caffarelli-Vazquez model of a porous media (degenerate)  diffusion model given by a fractional potential pressure law
\cite{Caffa}. Under some conditions, they  found
mass-preserving, nonnegative weak solutions of the equation satisfying energy estimates for the Cauchy problem. Moreover,  Caffarelli, Soria and Vazquez establish the H\"older regularity of such weak solutions for the case $s\neq1/2$ in \cite{CSV} and the case  $s=1/2$ has been proved in \cite{CV2} by Caffarelli and Vazquez.

\smallskip
A similar model was introduced at the same time  by Biler, Imbert, Karch  and Monneau (see \cite{BIK} \cite{BKM} and \cite{I}). A different approach to prove existence based on gradient flows has been developed by Lisini, Mainini and Segatti (see \cite{LMS}). Then the model has been generalized in \cite{STV1} \cite{STV2} \cite{STV3} \cite{STV4} \cite{STV5}. Uniqueness is still open in general, but under some truly restrictive regularity assumption is proven in \cite{ZXC}. 

\smallskip
On bounded domain, the Caffarelli-Vazquez model was studied by myself and Neves in \cite{Hua}. The main novelty of this work was to state how the boundary condition is considered. For $\frac{1}{2}<s<1$, the boundary condition is assumed in the sense of trace, and for $0<s\leq \frac{1}{2}$, we inspired in the definition of weak solutions for scalar conservation laws posed in bounded domains as proposed by Otto \cite{Otto} (see also \cite{Malek}, \cite{WN1}).

\smallskip
In another context, Nguyen and Vazquez \cite{NV} studied a similar model with a different approach in the definition of weak solution. Moreover they proved existence and smoothing effects.

\smallskip
In this paper, we focus in the (simplest) anisotropic degenerate case, that is, $\kappa(x,u)= u \ A(x)$, where the coefficients $(a_{ij})$, $i,j=1,\cdots,n$ describing the anisotropic, heterogeneous nature of the medium. 

%\smallskip
%In Section \ref{BN} we summarize without proofs the relevant material on Fractional Elliptic Operator in bounded domain .
% More precisely, we use the Dirichlet Spectral Fractional Elliptic (DSFE for short),
%which is denoted by $\mathcal{L}^s$ and defined as follows
%$$
%\mathcal{L}^{s}u(x):= \sum_{k=1}^\infty \lambda^s_k \ u_k \ \varphi_k(x), \quad u_k= \int_\Omega u(x) \varphi_k(x) dx, 
%$$ 
%where $\lambda_k> 0$, $k=1,2,\cdots,$ are the eigenvalues of the Elliptic operator $\mathcal{L}$ in $H^1_0(\Omega)$,
%and $\varphi_k$ are the corresponding normalized eigenfunctions.  

\smallskip
The main goal of this work is to state how boundary condition will be considered. In order to treat this part of the boundary, we follow an approach inspired by F. Otto \cite{Otto}. In method we propose, the boundary conditions, written as limits of integrals on $(0,T)\times\partial\Omega$ of a certain function. To this purpose, it is introduced a function $\Psi:[0,1]\times\partial\Omega\to\overline{\Omega}$ called $C^1$- admissible deformation (see Section \ref{AdDeform}). 

\smallskip
A simple explanation to use  the $C^1$-map $\Psi$ is the following. Consider the equation $\dive(uA(x)\nabla\mathcal{K}u)=0$ in $\Omega$, and $u=0$ on $\partial\Omega$. Multiply it by $\phi\in C^\infty_c(\R^n)$, integrate by part,  and from the boundary condition, we expect that
\begin{equation}
\int_{\partial\Omega}u(r)A(r)\nabla\mathcal{K}u(r)\cdot\nu(r)\phi(r)dr=0\label{motiv}
\end{equation}
where $\nu$ is the unit outward normal field on $\partial\Omega$. 
Notice that the existence of trace for $u$  does not necessary exist in the sense of traces in $H^s(\Omega)$. Moreover the trace for $u$ and $\nabla \mathcal{K}u\cdot\nu$ are mutually exclusive (see Remark \ref{REMMEX}). Then \eqref{motiv} is not well defined, to avoid this difficulty, it will be considered a simple modification, as follows
\begin{equation}
\ess\lim_{\tau \to 0^+}\int_{\partial\Omega_\tau} u(r) \ A(r)
\nabla \mathcal{K}u(r) \cdot \nu_\tau(r) \ \phi(\Psi_\tau^{-1}(r)) \,dr= 0,\nonumber
\end{equation}
where $\Psi_\tau(r)$ is a $C^1$-deformation, and $\nu_\tau$ is the unit outward normal field on $\partial \Omega_\tau=\Psi_\tau(\partial\Omega)$ (see Section \ref{AdDeform}).

\smallskip
On the other hand, we also show an equivalent definition of (weak) solutions as given by Definition \ref{DEFSOL}, more precisely an integral equivalent definition (see the Equivalence Theorem \ref{THEQUIVA}).

\medskip
After introducing the definition of weak solution to above problem, we study of existence of  solution in the proposed setting. We prove that the weak solution previously defined can be obtained as the limit of solution of regularized equation \eqref{FTPME}, to prove that we use energy estimates and apply the Aubin-Lions Compacteness Theorem.

\smallskip
On the other hand, an important talk is about the 
non-homogeneous Dirichlet boundary conditions.  
First, if a given boundary data $u_b\neq0$ is smooth enough to be considered as the restriction (in the sense trace in $H^s(\Omega)$) of a function  $u_b$ defined in $\Omega_T$ , then the strategy developed here follows right way with standard modifications. After that, some forcing terms appears, one of them is 
\begin{equation}
\dive(u_bA(x)\nabla\mathcal{K}u_b),\label{eq:ub}
\end{equation} 
thus to make sense \eqref{eq:ub}, it is necessary that $u_b\in D\big(\mathcal{L}^{(1-s)/2}\big)$, (see Definition \ref{DEFSOL}), 
but this is not necessary true, since $u_b\neq0$ on the boundary (see the counterexample in \cite{HAJPSR}). To avoid this difficulty, it have to use  
the fractional operators with inhomogeneous boundary conditions as defined in \cite{HAJPSR}. 

\smallskip
Finally, we stress that the uniqueness property is not established in this paper. In fact, it seems to be open even if for the Cauchy problem. Somehow, the ideas from scalar conservation laws could be useful, more precise, the doubling of variables of Kru\v zkov \cite{K}.

%%%%%%%%%%%%%%%%%%%%%%%%%%%
\section{Preliminaries} 
\label{BN}
%%%%%%%%%%%%%%%%%%%%%%%%%%%
In this section, we review some  results of Dirichlet spectral fractional elliptic (DSFE for short) and admissible deformation.  We mainly provide the proofs of the new results, in particular we stress Proposition \ref{proKuHu}. One can refer to \cite{MSireVazquez}, \cite{LCPRS}, and \cite{Hua}  for an introduction.

%%%%%%%%%%%%%%%%%%%%%%
\medskip
Let $\Omega$ be a bounded open set in $\R^n$. We denote by $\clg{H}^\theta$ the $\theta$-dimensional Hausdorff measure, and $\big(L^2(\Omega)\big)^n$  is the Cartesian  product of $L^2(\Omega)$ $n$-times.

\subsection{Dirichlet Spectral Fractional Elliptic}
%%%%%%%%%%%%%%%%%%%%%%%%%%%%%%%
Here and subsequently, $\Omega \subset \R^n$ is a bounded open set with $C^2$-boundary $\partial \Omega$. 
We are mostly interested in fractional powers of a strictly positive self-adjoint 
operator defined in a domain, which is dense in a (separable) Hilbert space. Therefore,
we are going to consider hereupon the operator $\mathcal{L}u:=-\dive(A(x)\nabla u)$ with homogeneous Dirichlet data, where $A(x)=(a_{ij}(x))_{n\times n}$ is a matrix, such that $a_{ij}\in C^\infty(\bar{\Omega})$ ($i,j=1,\cdots,n$) and satisfy the uniform elliptic condition
\begin{equation}
\Lambda_1|\xi|^2\leq \sum_{i,j=1}^na_{ij}(x)\xi_i\xi_j\leq \Lambda_2|\xi|^2, \label{unifellip}
\end{equation}
for all $\xi\in\R^n$ and a.e. $x\in\Omega$, for some ellipticity constant $0<\Lambda_1\leq\Lambda_2$. Moreover, the coefficients  are symmetric $a_{ij}(x)=a_{ji}(x),\,i,j=1,\cdots,n$,  bounded and measurable in $\Omega$.

\medskip
Due to well-known the elliptic operator $\mathcal{L}$  is nonnegative and selfadjoint in $H^1_0(\Omega)$, therefore from spectral theory, there exists a complete orthonormal 
basis $\{\vp_k\}^{\infty}_{k=1}$ of $L^2(\Omega)$,  where $\vp_k \in H_0^1(\Omega)$
are eigenfunction corresponding to eigenvalue $\lambda_k$ for each $k\geq 1$, moreover
$$
    0< \lambda_1<\lambda_2\leq\lambda_3\leq \cdots \leq \lambda_k  \leq \cdots, \quad \text{$\lambda_k \rightarrow \infty$
     as $k \longrightarrow \infty$}.
$$
Therefore the operator $\mathcal{L}$ and its the domain $D(\mathcal{L})$  could be rewrite as follow
$$
\begin{aligned}
   D(\mathcal{L})&= \left\lbrace u \in L^2(\Omega); \ \sum_{k= 1}^\infty \lambda_k^2 \, |\langle u, \vp_k \rangle|^2 < \infty \right\rbrace, 
   \\
   \mathcal{L} u&= \sum_{k= 1}^\infty \lambda_k \, \langle u, \vp_k \rangle \ \vp_k, 
   \quad \text{for each  $u \in D(\mathcal{L})$}.
\end{aligned}
$$
\begin{remark}
Since $\partial\Omega$ is $C^2$, it follows that $\varphi_k\in C^{\infty}(\Omega)\cap C^2(\ol{\Omega})$,
(see \cite{DGNST}, p. 214)
 and $D(\mathcal{L})= H^2(\Omega) \cap H^1_0(\Omega)$
(see \cite{DGNST}, p. 186)
. The former property, that is the regularity of the eigenfunctions $\varphi_k$, help us to study the regularized problem \eqref{FTPME} and the second property is important in Proposition \ref{represendoDominio}.
%in particular item $(ii)$.
\end{remark}

\medskip
Now, from functional calculus, we have the following definition
\begin{definition}[DSFE]\label{DSFE}
Let $\Omega \subset \R^n$ is a bounded open set with $C^2$-boundary $\partial \Omega$. Consider the operator $\mathcal{L}u:=-\dive(A(x)\nabla u)$ with homogeneous Dirichlet data, where $A(x)=(a_{ij}(x))_{n\times n}$ is a symmetric matrix, such that $a_{ij}\in C^\infty(\bar{\Omega})$ 
$(i,j=1,\cdots,n)$ and satisfy the condition \eqref{unifellip}. For each $s \in (0,1)$, the DSFE $\mathcal{L}^{s}: D \big(\mathcal{L}^{s}\big) \subset L^2(\Omega) \to L^2(\Omega)$,  is defined as follow 
\begin{equation}
\label{DEFFRALAP}
\begin{aligned}
    \mathcal{L}^s u:&= \sum_{k= 1}^\infty \lambda^s_k \, \langle u, \vp_k \rangle \ \vp_k, 
\\
    D\big(\mathcal{L}^s\big)&= \left\lbrace u \in L^2(\Omega) : \; \sum^{\infty}_{k=1}
    \lambda^{2 s}_k \, \vert \langle u, \vp_k \rangle \vert^2<+\infty  \right\rbrace.
\end{aligned}
\end{equation}
Analogously, we can also define $\mathcal{L}^{-s}: D\big(\mathcal{L}^{-s}\big)\subset L^2(\Omega) \to L^2(\Omega)$ for $s \in (0,1)$.

\end{definition}
The next proposition generalize some properties of the $s$-fractional Laplacian in bounded domain. 
In particular, we observe that $D\big(\mathcal{L}^{-s}\big)= L^2(\Omega)$.
\begin{proposition}
\label{THMCARAC}
Let $\Omega \subset \R^n$ be a bounded open set, $s \in (0,1)$, and consider $\mathcal{L}^s$, and $\mathcal{L}^{-s}$ the 
operators defined above. Then, we have:

\begin{enumerate}
\item[$(1)$] $D\big(\mathcal{L}\big) \subset D\big(\mathcal{L}^{s}\big)$, thus $D\big(\mathcal{L}^{s}\big)$ is dense in $L^2(\Omega)$.

\item[$(2)$] For all  $u \in D\big(\mathcal{L}^{s}\big)$, there exists $\alpha> 0$ which is the coercivity constant of $\mathcal{L}$ and satisfies
\begin{equation}
   \langle \mathcal{L}^s u, u \rangle \geq \alpha^s \Vert u\Vert^2_{L^2(\Omega)}. \label{coercivity}
\end{equation}
Moreover, it follows that $(\mathcal{L}^s)^{-1}= \mathcal{L}^{-s}$, also
% \in \clg{L}(L^2(\Omega))$
 $\mathcal{L}^s$ and $\mathcal{L}^{-s}$ are self-adjoint.

\item[$(3)$] $D\big(\mathcal{L}^s\big)$ endowed with the inner product
\begin{equation}
      \langle u, v \rangle_{s}:=\langle u, v\rangle+\int_{\Omega} \mathcal{L}^{s} u(x) \; \mathcal{L}^{s} v(x) \, dx \nonumber
\end{equation}
is a Hilbert space. In particular the norm $\vert \cdot \vert_s$ is defined by
\begin{equation}
\label{eq:norm}
       \vert u \vert_s^2=\Vert u \Vert^2_{L^2(\Omega)}+\Vert \mathcal{L}^su \Vert^2_{L^2(\Omega)}.
\end{equation}

%\item[$(4)$]  For each $u \in D\big(\mathcal{L}^{s_2}\big)$ and $ 0 < s_1\leq s_2 < 1$, we have
%\begin{equation}
%    \Vert \mathcal{L}^{s_1} u \Vert^{s_2}_{L^2(\Omega)} \leq  \Vert \mathcal{L}^{s_2} u\Vert^{s_1}_{L^2(\Omega)} 
%    \, \Vert u\Vert^{s_2-s_1}_{L^2(\Omega)}.\label{As1s2}
%\end{equation}
%
%\item[$(5)$] If $u \in D\big(\mathcal{L}\big)$, then 
%\begin{equation}
%\lim_{s \to 0^+} \mathcal{L}^su= u, \lim_{s \to 1^-} \mathcal{L}^su= \mathcal{L}u %\quad \text{in $L^2(\Omega)$.}\nonumber
%\end{equation}
%Furthermore, \eqref{As1s2} holds for $s_1= 0$, and $s_2= 1$.
%
%\item[$(6)$] For any $\lambda> 0$, and $s \in [0,1]$, 
%the operator $I + \lambda \mathcal{L}^s$ is bijective. 
%Moreover, for any $v \in L^2(\Omega)$, the family 
%$\{v_\lambda\}$, $v_\lambda \in D(\mathcal{L}^s)$ defined by 
%$$
%    v_{\lambda}:=(I+\lambda \mathcal{L}^s)^{-1} v
%$$
%converges to  $v$ in $L^2(\Omega)$ as $\lambda \to 0$.
%
%\item[$(7)$] If $0\leq s_1 < s_2\leq 1$, then  
%$$
%   D\big(\mathcal{L}^{s_2}\big)\hookrightarrow D\big(\mathcal{L}^{s_1}\big), \quad %\text{and $ D\big(\mathcal{L}^{s_2}\big)$ is dense in  
%   $D\big(\mathcal{L}^{s_1}\big)$}. 
%$$
\end{enumerate}
\end{proposition} 
\begin{proof}
The proof proceed analogously to the proposition 2.1 \cite{Hua}
\end{proof}

\medskip
Now, we state a Poincare's type inequality for the DSFE, and an equivalent 
norm for $D\big(\mathcal{L}^s\big)$.

\begin{corollary}[ Poincare's type inequality ]
\label{poincare}
Let $\Omega \subset \R^n$ be a bounded open set. Then for each $s> 0$, we have
%there exists a positive constant $C_\Omega$, such that
$$
    \Vert u \Vert_{L^2(\Omega)} \leq \lambda_1^{-s} \ \Vert \mathcal{L}^s u \Vert_{L^2(\Omega)}, 
    \quad \text{for all $u \in D\big(\mathcal{L}^s\big)$}.
$$
Moreover, the norm defined in \eqref{eq:norm} and
\begin{equation}
\label{equi}
  \Vert u \Vert_s^2:= \int_{\Omega}\vert \mathcal{L}^s u(x) \vert^2 \ dx
\end{equation}
are equivalent. 
\end{corollary}
\begin{remark}
As a consequence of the above results, we could consider the  inner product in $D\big(\mathcal{L}^s\big)$, as follow
\begin{equation}
     \langle u, v \rangle_{s}= \int_{\Omega} \mathcal{L}^s u(x) \ \mathcal{L}^s v(x) \ dx.
\end{equation}
\end{remark}

\medskip
Now, the aim  is to characterize (via interpolation) the space $D(\mathcal{L}^{s})$. 
To begin, we consider $u \in D\big(\mathcal{L}\big)$, then, since $\mathcal{L}^{1/2}$ is self-adjoint and from the definition of $\mathcal{L}$ we have
$$
\begin{aligned}
\int_{\Omega} |\mathcal{L}^{1/2} u(x)|^2 \ dx&=\int_{\Omega} \mathcal{L}^{1/2} u(x) \ \mathcal{L}^{1/2} u(x) \ dx=\int_{\Omega} \mathcal{L} u(x) \  u(x) \ dx
\\[5pt]
&=\int_{\Omega} A(x)\nabla u(x) \cdot \nabla u(x) \ dx.
\end{aligned}
$$

On the other hand, using the uniform elliptic condition and choosing $\xi=\nabla u$ in \eqref{unifellip}, and after that integrate over $\Omega$, we obtain
\begin{equation}
\Lambda_1 \int_{\Omega}|\nabla u(x)|^2 \ dx\leq \int_{\Omega} A(x)\nabla u(x) \cdot \nabla u(x) \ dx\leq \Lambda_2 \int_{\Omega}|\nabla u(x)|^2 \ dx,\label{eqnorequi}
\end{equation}
therefore
\begin{equation}
 \Lambda_1\Vert u \Vert_{H^1_0(\Omega)}^2\leq   \Vert \mathcal{L}^{1/2}u\Vert_{L^2(\Omega)}^2\leq \Lambda_2\Vert u \Vert_{H^1_0(\Omega)}^2,\label{normequi}
\end{equation}
which mean the norm $\Vert\cdot\Vert_{1/2}$ is equivalent to the norm $\Vert \cdot\Vert_{H^1_0(\Omega)}$.  Consequently, from the density of $D\big(\mathcal{L}\big)$ in $D\big(\mathcal{L}^{1/2}\big)$, and also in $H^1_0(\Omega)$, it follows that
$D\big(\mathcal{L}^{1/2}\big)= H^{1}_0(\Omega)$. Similarly, we have the following result: 

\begin{proposition}
\label{represendoDominio}
%\label{eq:dominequiv}
Let $\Omega \subset \Bbb{R}^n$ be a bounded open set. 

$i)$ If $s \in (0,1/2]$, then 
\begin{equation}
  D\big(\mathcal{L}^s\big)= \left\{
 \begin{array}{rcl}
      H^{2 s}(\Omega),&if&\;\;\;\;\;0 < s < 1/4,
      \\[5pt]
      H^{1/2}_{00}(\Omega),&if&\;\;\;\;\; s= 1/4,
      \\[5pt]
      H^{2 s}_0(\Omega),&if&\;\;\;\;\;1/4 <  s \leq 1/2.
\end{array}\right.\label{equalrepredominio}
\end{equation}

$ii)$ If $s \in (1/2,1)$, then 
\begin{equation}
 D\big(\mathcal{L}^{s}\big)= \big[H^2(\Omega)\cap H^1_0(\Omega), H^1_0(\Omega)\big]_{1-\theta}\label{eq:dominequiv},
\end{equation}
where $\theta= 2 s -1$.
Moreover, $D\big(\mathcal{L}^{s}\big)\subset H^{2s}(\Omega)\cap H^1_0(\Omega)$. 
\end{proposition}

%\begin{proof} It is enough to observe that, 
%$D\big((-\Delta_D)^{s/2}\big)= \big[H^1_0(\Omega), H^1_0(\Omega)\big]_{1-s}$,
%for each $s \in (0,1)$, and then apply the discrete version of J-Method for interpolation, see
%Bonforte, Sire, Vazquez \cite{}.
%\end{proof}
\begin{proof} The proof follows applying the discrete version of J-Method for interpolation, see \cite{MSireVazquez} and also \cite{Grubb}.
% A different prove of the item (i) by using extension problem of Stinga-Torrea see \cite{LCPRS} and \cite{PRSJLT}.
\end{proof}

Here and subsequently, we  denote for each $s \in (0,1)$ the operators: 
$$
   \mathcal{K}:= \mathcal{L}^{-s} \quad \text{and} \quad \mathcal{H}= \clg{K}^{1/2}:=\mathcal{L}^{-s/2} .
$$
Then, we consider the following 

\begin{proposition}
\label{proKuHu}
Let $\;\Omega \subset \R^n$ be a bounded open set.
\begin{enumerate}
\item[$(1)$] 
There exists a constant $C_{\Omega}>0$ such that if $u\in H^1_0(\Omega)$, then
$\nabla \clg{K} u \in \big(L^2(\Omega)\big)^n$ and
\begin{equation}
\label{nablaKunablau}
   \int_{\Omega}|\nabla\mathcal{K}u(x)|^2 \ dx
   \leq C_{\Omega}\int_{\Omega}|\nabla u(x)|^2 \ dx.
\end{equation}
Similarly, for each $u\in H^1_0(\Omega)$,
$\nabla \clg{H} u \in \big(L^2(\Omega)\big)^n$ and
\begin{equation}
\label{nablaHunablau}
   \int_{\Omega}|\nabla\mathcal{H}u(x)|^2 \ dx
   \leq C_{\Omega}^{1/2} \int_{\Omega}|\nabla u(x)|^2 \ dx.
\end{equation}

\item[$(2)$] If $u\in H^1_0(\Omega)$, then
\begin{equation}
\label{pro:gradKugradu}
 \Lambda_1\int_{\Omega}\vert \nabla \mathcal{H} u \vert^2 \ dx\leq   \int_{\Omega}A(x)\nabla \mathcal{K} u\cdot \nabla u \ dx 
 \leq \Lambda_2\int_{\Omega}\vert \nabla \mathcal{H} u \vert^2 \ dx,
\end{equation}
%where it is not used that: $\nabla$ commutes with $\mathcal{K}$.
\end{enumerate}
\end{proposition}

\begin{proof}
Since $u\in H^1_0(\Omega)$, it is enough to consider $u \in C^{\infty}_c(\Omega)$, 
and then apply a standard density argument. 

\smallskip
To show (1), we use the equivalence norm \eqref{normequi} or \eqref{eqnorequi}.
 Then, we have
$$
\begin{aligned}
\int_{\Omega}|\nabla\mathcal{K}u(x)|^2dx&\leq \Lambda^{-1}_1  \int_{\Omega}|\mathcal{L}^{1/2}\mathcal{K}u(x)|^2dx=\Lambda^{-1}_1 \sum_{k=1}^{\infty}\lambda_k\vert \langle \mathcal{K} u, \vp_k \rangle \vert^2
\\[5pt]
&=\Lambda^{-1}_1 \sum_{k=1}^{\infty}\lambda_k\vert\lambda_k^{-s}\langle u, \vp_k \rangle\vert^2\leq \Lambda^{-1}_1  \lambda_1^{-2s}\sum_{k=1}^{\infty}\lambda_k\vert\langle u, \vp_k \rangle\vert^2
\\[5pt]
&=\Lambda^{-1}_1 \lambda_1^{-2s}\int_{\Omega}|\mathcal{L}^{1/2} u(x)|^2dx
\leq\Lambda^{-1}_1\Lambda_2 \ \lambda_1^{-2s}\int_{\Omega}|\nabla u(x)|^2dx< \infty,
\end{aligned}
$$ 
and analogously for $\nabla \clg{H} u$.

\medskip
Now, we prove (2). First, we integrate by parts to obtain
$$
    \int_{\Omega}A(x)\nabla \mathcal{K}u(x)\cdot\nabla u(x)\,dx
   =\int_{\Omega}-\dive(A(x)\nabla\mathcal{K}u(x))u(x)\,dx
   =\int_{\Omega}\mathcal{L}^{1-s}u(x)u(x)\,dx,
$$
where we have used the definition of $\mathcal{K}u$.
Due to the $\mathcal{L}^{1-s}$ being self-adjoint (Proposition \ref{THMCARAC}(2) ), it follows that
$$
\int_{\Omega}A(x)\nabla \mathcal{K}u(x)\cdot\nabla u(x)\,dx=\int_{\Omega}|\mathcal{L}^{(1-s)/2}u(x)|^2\,dx.
$$
Therefore, using the equivalence norm  \eqref{normequi} together with the definition of $\mathcal{H}u$, we have
$$
 \Lambda_1\int_{\Omega}\vert \nabla \mathcal{H} u(x) \vert^2 \ dx\leq   \int_{\Omega}A(x)\nabla \mathcal{K} u(x)\cdot \nabla u(x)dx \leq \Lambda_2\int_{\Omega}\vert \nabla \mathcal{H} u(x) \vert^2 \ dx,
$$
\end{proof}
\begin{remark}\label{remarkAx}
Under the above assumptions,  and by a similar arguments, we obtain that  $\clg{K} u \in H^{1+2 s}(\Omega) \cap H^1_0(\Omega)$ and
\begin{equation}
 \Lambda_1\int_{\Omega}\vert \nabla \mathcal{H} u(x) \vert^2 \ dx\leq   \int_{\Omega}A(x)\nabla u(x)\cdot \nabla  \mathcal{K}u(x)dx\leq \Lambda_2\int_{\Omega}\vert \nabla \mathcal{H} u(x) \vert^2 \ dx, \nonumber
\end{equation}
for all $u\in H^1_0(\Omega)$.
\end{remark}
% Indeed, it is enough to show that $(-\Delta_D)^s \nabla \clg{K} u \in L^2(\Omega)$.
%Finally, as mentioned at the introduction, 
\subsection{Heat Semigroup Formula}
There are another ways of defining fractional elliptic operator, which turn out to be equivalent to DSFE. 
Here, we recall the Heat Semigroup formula, and address \cite{LCPRS} for a complete description. 

\medskip
First, given a function $u= \sum_{k=1}^{\infty}u_k\varphi_k$ in 
$L^2(\Omega)$, the weak solution $v(t,x)$ of the IBVP
$$
\left\lbrace\begin{aligned}
   v_t+\mathcal{L}v=0,\;\;\;\;\;\;\;\;&\mbox{in}\;\Omega \times (0,+\infty),
   \\[3pt]
   v(x,t)= 0,\;\;\;\;\;\;\;\;&\mbox{on}\;\partial\Omega\times[0,+\infty),
   \\[3pt]
   v(x,0)=u(x),\;\;\;&\mbox{in}\;\Omega
\end{aligned}\right.\nonumber
$$
is given by 
$$
v(x,t)= e^{-t\mathcal{L}}u(x)=\sum_{k=1}^{\infty}e^{-t\lambda_k}u_k\varphi_k(x).
$$
In particular, $v\in L^2((0,\infty);H^1_0(\Omega))\cap C([0,\infty);L^2(\Omega))$ and $\partial_t v\in L^2((0,\infty);H^{-1}(\Omega))$. 

The following Lemma express in a different way the definition of DSFE.
\begin{lemma}
\label{lem:semihet}
Let $\Omega\subset \Bbb{R}^n$ be a bounded open set, and $0<s<1$.
\begin{enumerate}
\item[$(1)$] If $u\in D\big(\mathcal{L}^{s}\big)$, then
$$
\mathcal{L}^{s}u= \dfrac{1}{\Gamma(-s)}\int_0^{\infty}(e^{-t\mathcal{L}}u-u)\dfrac{dt}{t^{1+s}} \quad \text{in $L^2(\Omega)$}.
$$
More precisely, if $w\in L^2(\Omega)$, them
$$
\left\langle \mathcal{L}^{s}u,w  \right\rangle_{L^2(\Omega)}=\dfrac{1}{\Gamma(-s)}\int_0^{\infty}\left(\left\langle e^{-t\mathcal{L}}u,w  \right\rangle_{L^2(\Omega)}-\left\langle u,w  \right\rangle_{L^2(\Omega)}\right)\dfrac{dt}{t^{1+s}}
$$

\item[$(2)$] If $u\in L^{2}(\Omega)$, then
$$
\mathcal{L}^{-s} u= \dfrac{1}{\Gamma(s)}\int_0^{\infty}e^{-t\mathcal{L}} u \dfrac{dt}{t^{1-s}} \quad \text{in $L^2(\Omega)$}.
$$
\end{enumerate}
\end{lemma}

\begin{proof}
An excellent reference is the paper by Caffarelli and Stinga \cite{LCPRS}, see also \cite{Hua}.

The main basic idea of the proof is based on the following observation.
%(Sketch) First we recall that, the Gamma function is
%defined for $x> 0$, and for $-k< x< -k+1$ $(k \in \Bbb{N})$ respectively by
%$$
%    \Gamma(x)= \int_0^{\infty}t^{x-1}e^{-t}\,dt, \quad
%    \Gamma(x)= \int_0^{\infty}t^{x-1}\left[e^{-t}-\sum_{j=0}^{k-1}\dfrac{(-t)^j}%{j!} \right]\,dt.
%$$
%
For  any $\lambda>0$ and $0<s<1$ we have
\begin{align}
\lambda^{-s}&=\dfrac{1}{\Gamma(s)}\int_0^{\infty}e^{-t\lambda}\dfrac{dt}{t^{1-s}},\nonumber
\\[5pt]
\lambda^s\;\;&=\dfrac{1}{\Gamma(-s)}\int_0^{\infty}\left(e^{-t\lambda}-1 \right)
\dfrac{dt}{t^{1+s}},\nonumber
\end{align}
Now, from definition \eqref{DEFFRALAP}, and Fubini's Theorem, the proof follows.
\end{proof}

%%%%%%%%%%%%%%%
\subsection{Admissible Deformation}\label{AdDeform}
%%%%%%%%%%%%%%%
Let us fix here some notation and background used in this paper, we first consider the notion of $C^1$-(admissible) deformations, which is used to give the correct notion of traces. One can refer to \cite{NPS}.
\begin{definition}\label{AdmDefo}
Let $\Omega \subset \R^{n}$ be an open set.
A $C^1$-map $\Psi:[0,1] \times \partial \Omega\to \ol{\Omega}$ is said a $C^1$ 
admissible deformation,
when it satisfies the following conditions:
\begin{enumerate}
\item[$(1)$] For all $r\in\partial\Omega$, $\Psi(0,r)=r$.

\item [$(2)$] The derivative of the map $[0,1] \ni \tau \mapsto \Psi(\tau, r)$
at $\tau= 0$ is not orthogonal
to $\nu(r)$, for each $r\in \partial\Omega$.
\end{enumerate}
\end{definition}
Moreover, for each $\tau\in[0,1]$, we denote:
$\Psi_\tau$ the mapping from $\partial\Omega$ to $\ol{\Omega}$, given by $\Psi_\tau(r):=\Psi(\tau,r)$;
$\partial\Omega_\tau= \Psi_\tau(\partial\Omega)$;
$\nu_\tau$ the unit outward normal field in $\partial \Omega_\tau$. In particular,
$\nu_0(x)=\nu(x)$ is the unit outward normal field in $\partial\Omega$.

\begin{remark}
It must be recognized that domains with $C^2$ boundaries 
always have $C^1$ admissible deformations. Indeed, 
it is enough to take $\Psi(\tau,r)= r - \epsilon \tau \nu(r)$ for
sufficiently small $\epsilon > 0$. From now on, we say
$C^1$-deformations for short.
\end{remark}
%%%%%%%%%%%%%%%%%%%%%%%%%%%%%%%%%%%%%%%%%%%%%%%%%%%%%%%%%%%%%%%%%%%

\medskip
Now, we state the following Lemma, which will be useful to the define the level set function associated with the $C^1$-deformation $\Psi$.
\begin{lemma}\label{levelfun}
Let $\Omega \subset \R^{n}$ be an open set with $C^1$-boundary $\partial\Omega$ and the   $C^1$  deformation $\Psi:[0,1] \times \partial \Omega\to \ol{\Omega}$, then there exist $m\in \N$, $V_i\subset\R^n$ and $h_i\in C^1(V_i)$ $(i=1, \cdots, m)$, such that 
$$
x\in \partial \Omega_{\tau}\cap V_i\Rightarrow h_i(x)=\tau
$$
for all $i=1, \cdots, m$.
\end{lemma}
\begin{proof}
Since $\Omega \subset \R^{n}$ be an open set with $C^1$ boundary. Then, for each $x \in \partial\Omega$ there exists a neighbourhood $W$ of $x$ in $\R^n$, an open set $U\subset\R^{n-1}$ and a $C^1$ diffeomorphism mapping 
$\zeta: U \to \partial\Omega\cap W$.

\medskip
 On the other hand, we define 
 $\psi:[0,1]\times U\longrightarrow \overline{\Omega}$ by
\begin{equation}
\psi(\tau,y):=\Psi(\tau,\zeta(y)),\nonumber
\end{equation}
which is a $C^1$ function, due to $\xi$ and $\Psi$ are $C^1$ . Moreover from the item (2) of the definition \ref{AdmDefo}, we have the Jacobian of $\psi$ in $(0,y)$, satisfies
$$
J\psi(0,y)=J[\zeta](y)\left| \partial_\tau\Psi(0,\zeta(y))\cdot \nu(\zeta(y))\right|>0,
$$
for all $y\in U$. Then, applying the Inverse Function Theorem and passing to a 
smaller neighbourhood if necessary (still denoted by $U$),
there exists $\varrho>0$ such that, the function 
$\psi:[0,\varrho)\times U\longrightarrow \overline{\Omega}$ is a $C^1$ diffeomorphism onto its image.

\medskip
At the same time, since $\partial\Omega$ is compact, we can find finitely many points $x_i\in \partial\Omega$, 
corresponding sets $W_i\subset \R^n$; $U_i\subset\R^{n-1}$ 
and functions $\gamma_i\in C^1(U_i)$ $( i=1,\cdots, m)$, such that $\partial\Omega\subset \bigcup_{i=1}^{m} W_i$ and
$$
\gamma_i:U_i\longrightarrow\partial\Omega\cap W_i,
$$
moreover, there exists $\varrho_i> 0$, $(i=1,\ldots,m)$, such that, $\psi_i:[0,\varrho_i)\times U_i\longrightarrow \overline{\Omega}$ 
is a $C^1$ diffeomorphism onto its image, where $\psi_i(\tau,y):=\Psi(\tau,\gamma_i(y))$.

\medskip
Finally, we consider $\varrho= \min\{\varrho_i;\,i=1,\cdots, m\}$. Define $V_i:=\Psi([0,\varrho)\times \gamma_i(U_i))$ and
$h_i:V_i \to [0,\varrho)$, as follow
$$
h_i(x):=\pi_1\circ \psi_i^{-1}(x),\;\;\;\;x\in V_i,
$$
where $\pi_1:\R\times\R^{n-1}\to\R$, given by $\pi_1(a,b)=a$. In particular, if $x\in \partial\Omega_\tau\cap V_i$, we obtain that  $h_i(x)=\tau$.

\end{proof}
As a consequence of the above Lemma, we define the \emph{level set function associated with the $C^1$-deformation} $\Psi$, that is to say, the function 
%$$
%h: \bigcup_{i=1}^{m} V_i\to [0,\varrho)
%$$
$$
h: \overline{\Omega}\to [0,\varrho)
$$
by setting $h(x)=h_i(x)$, if $x\in V_i$ and $h(x)=\varrho$, for $x\in \overline{\Omega}\setminus \bigcup_{i=1}^{m} V_i$ , which is clearly a $C^1$ function. Moreover,
%since the function $\psi_i$ is a $C^1$ diffeomorphism, 
we have that $\nabla h(x)\neq 0$ for all $x\in \bigcup_{i=1}^{m} V_i$, and also  $\nabla h(r)$ is parallel to $\nu_\tau(r)$ on $\partial\Omega_\tau$.

%%%%%%%%%%%%%%%%%%%%%%%%%%%%%%%%%%%%%%%%%%%%%%%%%%%%%%%%%%%%%%%%%%%%%%
%%%%%%%%%%%%%%%%%%%%%%%%%%%%%%%%%%%%%%%%%%%%%%%%%%%%%%%%%%%%%%%%%%%%%

\medskip
To follow, we define some auxiliary functions,
which are important to show existence of solutions of the IBPV \eqref{FTPME}. 
\begin{enumerate}
\item[$(1)$] Without loss of generality, we may assume $\varrho=1$ (define in lemma \ref{levelfun}), and define
$$
   s(x):= \left \{ 
         \begin{aligned}
           h(x), &\quad \text{if $x \in \Omega$},
            \\[5pt]
           - h(x),& \quad \text{if $x \in \Bbb{R}^n \setminus \Omega$}.           
         \end{aligned}\right.
$$
\item[$(2)$] For each $k \in \Bbb{N}$, and all $x\in \R^n$,
define $\xi_k$ by
\begin{equation}
\xi_{k}(x):= 1-{\exp} \left( -k \ s(x)\right).\label{eq:xi}
\end{equation}
\end{enumerate}

%%%%%%%%%%%%%%%%%%%%%%%%%%%%%%%%%%%%%%%%%%%%%%%%%%%%%%%%%%%%%%%%%%%%%%
%%%%%%%%%%%%%%%%%%%%%%%%%%%%%%%%%%%%%%%%%%%%%%%%%%%%%%%%%%%%%%%%%%%%%%%
%%%%%%%%%%%%%%%%%%%%%%%%%%%%%%%%%%%%%%%%%%%%%%%%%%%%%%%%%%%%%%%%%%%%%%
%%%%%%%%%%%%%%%%%%%%%%%%%%%%%%%%%%%%%%%%%%%%%%%%%%%%%%%%%%%%%%%%%%%%%
%\medskip
%Now, we define a level set function $h$ associated with the deformation $\Psi_\tau$.
%For $\delta> 0$ sufficiently small we define
%$$
%   h(x):= \left \{ 
%         \begin{aligned}
%           \min\{\tau,\delta\}, &\quad \text{if $x \in \Omega$},
%            \\[5pt]
%           - \min\{\tau,\delta\},& \quad \text{if $x \in \Bbb{R}^n \setminus \Omega%$}.           
%         \end{aligned}\right.
%$$
%Then, we define $\xi_k$ for each $k \in \Bbb{N}$, and all $x\in \R^n$, by
%\begin{equation}
%\xi_{k}(x):= 1-{\exp} \left( -k \ h(x) \right).\label{eq:xi}
%\end{equation}

\begin{lemma}
\label{Lemma:xik}
Let $\Omega \subset \R^n$ be an open bounded domain with $C^2$ boundary.
Then, it follows that:
\begin{enumerate}
\item[$(1)$] The function $s(x)$ is Lipschitz continuous in $\Bbb{R}^n$, and $C^1$ on the
closure of $\left\lbrace x \in \Bbb{R}^n: |s(x)| < \delta \right\rbrace$.

\item[$(2)$] The sequence $\{\xi_{k}\}$ satisfies
\begin{equation}
 \lim_{k\to +\infty}\int_{\Omega}|1-\xi_k|^2dx=0,
  \quad \text{and} \quad
  \lim_{k\to +\infty}\int_{\Omega}|\nabla\xi_k|^2dx=0.
  \label{eq:limvarep}
\end{equation}
\end{enumerate}
\end{lemma}
\begin{proof}
This Lemma is an extension of the result obtain in section 2.8 of M\'alek, Necas, Rokyta and Ruzicka \cite{Malek}, p. 129.
\end{proof}

To finish this section, let us consider the following   
\begin{enumerate}
\item[(1)] Let a non-negative function $\gamma \in C^1_c(\R)$, with support contained in $[0,1]$, 
such that, $\int \gamma(t) dt= 1$. Then, we consider the sequences $\{\delta_j\}_{j\in \Bbb{N}}$,
and $\{H_j\}_{j\in \Bbb{N}}$, defined by
$$
   \delta_j(t):= j \ \gamma(j t), \quad H_j(t):= \int_0^t \delta_j(s) \ ds.
$$
Thus, for each $j \geq 1$, $H'_j(t)= \delta_j(t)$, and clearly the sequence $\{H'_j\}$  
converges as $j \to \infty$ to the Dirac $\delta$-measure in $\mathcal{D}'(\Bbb{R})$,
while the sequence $\{H_j\}$  converges pointwise to the Heaviside function 
$$
   H(t)= \left \{
   \begin{aligned}
   1, &\quad \text{if $t \geq 0$},
   \\
   0, &\quad \text{if $t< 0$}.
   \end{aligned}
   \right.  
$$

\item[(2)] 
\medskip 
Let $\Psi$  a $C^1$-admissible deformation  and   $\partial \Omega$ is $C^2$. Then for any point $x\in\partial\Omega$ 
there exists a neighbourhood $W$ of $x$ in $\R^n$, an open set $U\subset\R^{n-1}$ and a $C^2$ mapping 
$\zeta:U \to \partial\Omega\cap W$, which is a $C^1-$diffeomorphism. Moreover, it satisfies
\begin{equation}
\lim_{\tau\to 0} J[\Psi_\tau\circ \zeta]=J[\zeta]\;\;\;\;in\;\;C(U),\nonumber
\end{equation}
where $J[\cdot]$ is the Jacobian. Furthermore $J[\Psi_\tau]$  defined by 
\begin{equation}
J[\Psi_\tau](r):=\dfrac{J[\Psi_\tau\circ\zeta](\zeta^{-1}(r))}{J[\zeta](\zeta^{-1}(r))},\label{eq:Jpsi}
\end{equation}
satisfies $J[\Psi_\tau]\to 1$ uniformly as $\tau \to 0$.
\end{enumerate}
%%%%%%%%%%%%%%%%%%%%%%%%%%%%
%%%%%%%%%%%%%%%%%%%%%%%%%%%%
\section{Initial Boundary Value Problem} 
\label{SOLVAB}
%%%%%%%%%%%%%%%%%%%%%%%%%%%%
%%%%%%%%%%%%%%%%%%%%%%%%%%%%
Here we give a definition, which stablish how the boundary condition will be consider for the equation \eqref{FTPME}. We also state an equivalent definition of weak solution 
( Equivalent Theorem \ref{THEQUIVA} ). 

\subsection{Definition of weak solution}
We seek for a suitable (weak) solution $u(t,x)$ defined in $\Omega_T$, in this way the next definition tells us in which sense $u(t,x)$ is a solution to the IBVP
\eqref{FTPME}.
\begin{definition}
\label{DEFSOL}
Given an initial data $u_0 \in L^\infty(\Omega)$ and $0< s <1$, a function 
$$
    u\in L^2\left((0,T); D\big(\mathcal{L}^{(1-s)/2}\big)\right) \cap L^\infty(\Omega_T)
$$
is called a weak solution of  the IBVP \eqref{FTPME}, 
when  $u(t,x)$ satisfies:
\begin{enumerate}
\item[$1.$] The integral equation: For each  $\phi \in C^{\infty}_c(\Omega_T)$
\begin{equation}
    \iint_{\Omega_T} u(\partial_t\phi-A(x)\nabla \mathcal{K}u\cdot \nabla \phi) \ dxdt= 0.
\label{eq:solfrac}
\end{equation}
\item[$2.$] The initial condition: For all $\zeta\in L^1(\Omega)$
\begin{equation}
   \ess \lim_{t \to 0^+} \int_\Omega u(t,x) \zeta(x)\, dx=  \int_\Omega  u_0(x) \zeta(x)\, dx.\label{eq:initalcondition}
\end{equation} 
\item[$3.$] The boundary condition: 
For each $\gamma \in C^{\infty}_c((0,T)\times\R^n)$ 
\begin{equation}
\label{eq:boundcondition}
\hspace{-0.7cm}\ess\!\!\lim_{\tau \to 0^+}\int^T_0\!\!\!\int_{\partial\Omega_{\tau}}\!\!\!u(t,r)A(r)\nabla\mathcal{K}u(t,r)
    \cdot\nu_\tau(r) \gamma(t,\Psi^{-1}_\tau(r))\,d\clg{H}^{n-1}(r)dt= 0.
\end{equation}
%where $\Psi_\tau(r)$ is a $C^1$-deformation, 
%and $\nu_\tau$ is the unit outward normal field in $\partial \Omega_\tau$. 
\end{enumerate}
\end{definition}
%%%%%%%%%%%%%%%%%%%%%%%%%%%%%%%%%%%%%%%%%%
%\begin{remark}
%Note that in \eqref{eq:solfrac}, the test function $\phi\in C^{\infty}_c(\Omega_T)$ does not see the boundary of $\Omega$, so we do not have  information about the behaviour of the solution in the boundary. For this reason we incorporate the condition \eqref{eq:boundcondition}.
%, which expresses the sense of the boundary condition of \eqref{FTPME}.
%\end{remark}
%%%%%%%%%%%%%%%%%%%%%%%%%%%%%%%%%%%%%5
%\begin{remark}
%The boundary condition \eqref{eq:boundcondition} was motivated by the following fact. If we multiply  by the test function which see the boundary  $\phi\in C^{\infty}_c((0,T)\times\R^n)$ in the equation \eqref{FTPME} and integrate by part, we get
%$$
%\begin{aligned}
%    \iint_{\Omega_T} u(t,x)\partial_t\phi(t,x) \ dxdt \ -&\iint_{\Omega_T}u(t,x)A(x)\nabla \mathcal{K}u(t,x)\cdot  \nabla \phi(t,x) \ dxdt \ +
%\\[5pt]
%\int_0^T\int_{\partial\Omega} &u(t,r)A(r)
%\nabla \mathcal{K}u(t,r) \cdot \nu(r)\phi(t,r) \,d\clg{H}^{n-1}(r)dt=0,
%\end{aligned}
%$$
%since $u=0$, we expect that
%\begin{equation}
%\int_0^T\int_{\partial\Omega} u(t,r)A(r)
%\nabla \mathcal{K}u(t,r) \cdot \nu(r)\phi(t,r) \,d\clg{H}^{n-1}(r)dt=0.\label{eqinspira}
%\end{equation}
%But \eqref{eqinspira} is not necessarily true, because the trace of $\nabla \mathcal{K}u(t,r) \cdot \nu(r)$ could not exist. For this reason we use the $C^1$-deformation $\Psi$ then we obtain \eqref{eq:boundcondition}.
%\end{remark}
%%%%%%%%%%%%%%%%%%%%%%%%%%%%%%%%%%%%%%%%%%%%%%%%%%%%%%%%%%%
\begin{remark}
\label{REMMEX}
Given $u \in L^2\left((0,T);D\big(\mathcal{L}^{(1-s)/2}\big)\right)$
the limit in the left hand side of \eqref{eq:boundcondition}, a priori,
does not necessarily exist. Indeed, the existence of trace for 
$u$ and $\nabla \mathcal{K} u \cdot \nu$ are mutually exclusive. 
For instance, if $0<s<1/2$ then from Proposition \ref{represendoDominio}, it follows that 
$u \in L^2\left((0,T);H^{1-s}_0(\Omega)\right)$, which implies that $u$ has trace on $\partial\Omega$, moreover $u=0$ on $(0,T)\times\partial\Omega$, contrarily 
%On the other hand, 
$\mathcal{K}u\in L^2\left((0,T);H^{1+s}(\Omega)\cap H^1_0(\Omega)\right)$,
which means that, $\nabla\mathcal{K}u \cdot \nu$ does not have trace on $\partial\Omega$.
Vice versa result for $1/2 \leq s < 1$. 
\end{remark}

\medskip
However, if $u\in L^2\left((0,T); D\big(\mathcal{L}^{(1-s)/2}\big)\right) \cap L^\infty(\Omega_T)$ 
and satisfies \eqref{eq:solfrac}, then the essential limit in \eqref{eq:boundcondition} exist, in particular the boundary condition
 makes sense. Analogously, the initial conditional 
\eqref{eq:initalcondition}. 
%To show that, we first consider the following lemmas.
\begin{lemma}
\label{exisintegral}
Let $u\in L^2\left((0,T);D\big(\mathcal{L}^{(1-s)/2}\big)\right)\cap L^{\infty}(\Omega_T)$, with $s \in (0,1)$.
Then, for each function $\gamma\in C^{\infty}_c((0,T)\times\R^n)$ and any $C^1$-deformation $\Psi$ 
$$
   \int^T_0 \!\!\! \int_{\partial\Omega_\tau} u(t,r)A(r)\nabla \mathcal{K}u(t,r)\cdot \nu_\tau(r) \ \gamma(t,\Psi^{-1}_\tau(r))\,d\clg{H}^{n-1}(r)dt
$$
exists for a.e. $\tau>0$ small enough.
\end{lemma}
\begin{proof}
First, due to $u\in L^2\left((0,T);D\big((-\Delta_D)^{(1-s)/2}\big)\right)$, the integral
$$
    \int_0^T\int_{Im(\Psi)} u(t,x) \nabla \mathcal{K}u(t,x) \cdot \nabla h(x) \  \gamma(t,\Psi_{h(x)}^{-1}(x)) \ dxdt
$$
exists, where $h$ is the level set function associated with the deformation $\Psi_\tau$, which is defined in Section \ref{BN}. 
Hence applying the Coarea Formula for the function $h$, we obtain
\begin{equation}
\begin{aligned}
& \int_0^T\int_{Im(\Psi)} u(t,x) \nabla \mathcal{K}u(t,x) \cdot \nabla h(x) \  \gamma(t,\Psi_{h(x)}^{-1}(x)) \ dxdt
\\[5pt]
 &=\int^1_0\int^T_0\int_{\partial\Omega_\tau}u(t,r)\nabla\mathcal{K}u(t,r)\cdot\nu_\tau(r) \gamma(t,\Psi_{\tau}^{-1}(r))\,d\clg{H}^{n-1}(r)dt\,d\tau.
 \end{aligned}\label{eq:exisintegral}
\end{equation}
Thus, we obtain from \eqref{eq:exisintegral} that
\begin{equation}
\int^T_0\int_{\partial\Omega_\tau}u(t,r)\nabla\mathcal{K}u(t,r)\cdot\nu_\tau(r) \gamma(t,\Psi_{\tau}^{-1}(r))\,d\clg{H}^{n-1}(r)dt
\label{eq:exisintegral1}
\end{equation}
exists for a.e. $\tau\in(0,1)$ and each $\gamma\in C^{\infty}_c((0,T)\times\R^n)$.
\end{proof}
%
%\begin{lemma}
%\label{equiboundary}
%Let $u\in L^2\left((0,T);D\big(\mathcal{L}^{(1-s)/2}\big)\right)\cap L^{\infty}(\Omega_T)$, with $s \in (0,1)$.
%If for each function $\gamma\in L^2((0,T)\times\partial\Omega)$ and any $C^1$-deformation $\Psi$ 
%\begin{equation}
%  \ess\!\!\lim_{\tau \to 0^+}\int_0^{T}\!\!\!\int_{\partial\Omega_\tau} u(t,r)A(r)\nabla\mathcal{K}u(t,r)
%  \cdot\nu_\tau(r)\gamma(t,\Psi^{-1}_\tau(r))\,d\clg{H}^{n-1}(r)dt\label{eq:esslimequi}
%\end{equation}
%exists, then
%$$
%    \ess\!\!\lim_{\tau \to 0^+}\int^T_0
%    \!\!\!\int_{\partial\Omega} u(t,\Psi_{\tau}(r))A(\Psi_{\tau}(r))\nabla \mathcal{K}u(t,\Psi_{\tau}(r))
%    \cdot\nu_\tau(\Psi_{\tau}(r))\gamma(t,r)\,d\clg{H}^{n-1}(r)dt
%$$
%exists and both limits are equal.
%\end{lemma}
%\begin{proof}
%See \cite{Hua}
%\end{proof}
%
%%%%%%%%%%%%
%

\medskip
To follow, we define some auxiliary set, which are important to show that the Definition \ref{DEFSOL} makes sense. Let $u\in L^2\left((0,T);D\big(\mathcal{L}^{(1-s)/2}\big)\right)\cap L^{\infty}(\Omega_T)$ be a function satisfying \eqref{eq:solfrac}, then consider the following sets:

\begin{enumerate}
\item[(1)] Let $\clg{E}$ be a countable dense subset of $C^{1}_c(\Omega)$.
For each $\zeta \in \clg{E}$, we define the set of full measure in $(0,T)$ by
$$
 \text{$E_\zeta:= \Big\{ t \in(0,T) / \, t $ is a Lebesgue point of $ I(t)= \int_{\Omega}u(t,x)\zeta(x)dx \Big\}$},
$$
and consider
$$
E:=\bigcap_{\zeta\,\in\, \clg{E}}E_\zeta,
$$
which is a set of full measure in $(0,T)$.
\item[(2)]  Let $\clg{F}$ be a countable dense subset of $C^{\infty}_c((0,T)\times\R^n)$.
For each $\gamma \in \clg{F}$, we define the set of full measure in $(0,1)$ by
$$
 \text{$F_\gamma= \Big \{ \tau\in(0,1) / \tau$ is a Lebesgue point of $J(\tau) \Big \}$}, 
$$
where
$$
J(\tau)=\int_0^{T}\int_{\partial\Omega_\tau}u(t,r)A(r)\nabla\mathcal{K}u(t,r)\cdot\nu_\tau(r)\gamma(t,\Psi_\tau^{-1}(r))\,d\clg{H}^{n-1}(r)dt,
$$
which makes sense thanks to Lemma \ref{exisintegral}. 
Moreover, we consider
$$
F:=\bigcap_{\gamma\,\in\, \clg{F}}F_\gamma,
$$
which is also a set of full measure in $(0,1)$. For more details see \cite{Hua}.
\end{enumerate}

\medskip
The next theorem ensures the existence of the essential limit \eqref{eq:initalcondition} and the boundary condition  \eqref{eq:boundcondition}
\begin{theorem}
\label{PROP10}
Let $u\in L^2\left((0,T);D\big(\mathcal{L}^{(1-s)/2}\big)\right)\cap L^{\infty}(\Omega_T)$ 
and assume that $u$ satisfies \eqref{eq:solfrac}, then:
\begin{enumerate}
\item[$(1)$] There exists a function $\bar{u}\in L^{\infty}(\Omega)$, such that 
% there exists a set $E\subset \R_+$ of full measure, which satisfies
\begin{equation}
\label{existinitial}
   \ess\!\! \lim_{t \to 0^+} \int_{\Omega} u(t,x)\zeta(x) dx=\int_{\Omega} \bar{u}(x)\zeta(x) dx,
\end{equation}
for each $\zeta \in L^1(\Omega)$.

\item[$(2)$] For each $\gamma \in C^{\infty}_c((0,T)\times\R^n)$, and any $C^1$-deformation $\Psi$, the
%exist a set $F\subset \R_+$ of full measure, such that 
\begin{equation}
\label{existboundary}
 \hspace{-0,4cm}\ess \!\!\lim_{\tau \to 0^+} \int^T_0\int_{\partial\Omega_{\tau}}u(t,r)A(r)\nabla\mathcal{K}u(t,r)
    \cdot\nu_\tau(r) \gamma(t,\Psi^{-1}_\tau(r))\,d\clg{H}^{n-1}(r)dt,
\end{equation}
exists. 
\end{enumerate}
\end{theorem}

\begin{proof}
1. To prove \eqref{existinitial}, let $\zeta \in \clg{E}$ and consider the set $E$ defined above. 
Then, for each $t \in E$, $\Vert u(t,\cdot) \Vert_{\infty}\leq C$. Thus we can find a
sequence $\{t_m\}$, $t_m \in E$, $m \in \Bbb{N}$, $t_m\to 0$ as $m \to \infty$ and a function $\bar{u} \in L^{\infty}(\Omega)$, such that
$u(t_m,\cdot) \to  \bar{u}(\cdot)$ weakly-* in $L^{\infty}(\Omega)$ as $m\to \infty$.

\medskip
If $c\in E$, then for large enough $m$, $t_m < c$. We fix such $t_m < c$ and set $\gamma_j(t)=H_j(t-t_m)-H_j(t-c)$, 
where the sequence $H_j(\cdot)$, $j\in \Bbb{N}$ is defined in Section \ref{BN}. Therefore, taking in \eqref{eq:solfrac}  
$\phi(t,x)=\gamma_j(t)\zeta(x)$, we have
$$
\iint_{\Omega_T}u(t,x)\gamma_j'(t)\zeta(x)dxdt=\iint_{\Omega_T}u(x)A(x)\nabla\mathcal{K}u\cdot\nabla\zeta(x)\,\gamma_j(t)\,dx\,dt.
$$
The above expression may be written as
\begin{equation}
\int_{0}^T\gamma_j'(t)\,I(t)dt=\int_0^T \gamma_j(t)\int_{\Omega}u(x)A(x)\nabla\mathcal{K}u\cdot\nabla\zeta(x)\,dx\,dt.\label{eq:existinitial1}
\end{equation}

\medskip
Then passing to the limit in (\ref{eq:existinitial1}) as $j \to \infty$, and 
taking into account that, $t_m,\;c$ are Lebesgue points of the function $I(t)$, 
also that $\gamma_j(t)$ converges pointwise 
to the characteristic function of the interval $[t_m,c)$, we obtain
$$
I(t_m)-I(c)=\int^{c}_{t_m}\int_{\Omega}u(x)A(x)\nabla\mathcal{K}u\cdot\nabla\zeta(x)\,dx\,dt,
$$
which implies in the limit as $m\to\infty$ that
$$
\int_{\Omega} \bar{u}(x)\zeta(x) dx-I(c)=\int^{c}_{0}\int_{\Omega}u(x)A(x)\nabla\mathcal{K}u(x)\cdot\nabla\zeta(x)\,dx\,dt,
$$
for all $c\in E$. Therefore, in view of the density of $\clg{E}$ in $L^1(\Omega)$, we have
$$
   \lim_{E \ni t \to 0} I(t)= \int_{\Omega} \bar{u}(x)\zeta(x) dx,
$$
for each $\zeta \in L^1(\Omega)$, which proves \eqref{existinitial}.

\medskip
2. Now, we show \eqref{existboundary}. 
Let $\gamma \in \clg{F}$, consider $F$, and define $S:= \Psi(F \times \partial \Omega)$. 
For $\tau_1,\,\tau_2 \in F$, with $\tau_1<\tau_2$, define $\zeta_j(\tau)= H_j(\tau-\tau_1)-H_j(\tau-\tau_2)$, $j\in\Bbb{N}$,
and take in \eqref{eq:solfrac} $\phi(t,x)$ defined by
$$
 \phi(t,x)=\left\lbrace
 \begin{aligned}
 \gamma(t,\Psi_{h(x)}^{-1}(x))\zeta_j(h(x)), &\quad \text{for $x \in S$}, 
 \\[5pt]
0\;, & \quad \text{for $x \in \Omega \setminus S$},
 \end{aligned}\right.
$$
where $h(x)$ is the level set associated with the deformation $\Psi_\tau$, which is defined in Section \ref{BN}. Then, we have
$$
\begin{aligned}
     &\int_0^T\int_{S} u(t,x) \ \partial_t\gamma(t,\Psi_{h(x)}^{-1}(x)) \ \zeta_j(h(x)) \ dxdt 
     \\[5pt]
     & \quad=\int_0^T\int_{S} u(t,x)A(x)\nabla \mathcal{K}u(t,x) \cdot \nabla\gamma(t,\Psi_{h(x)}^{-1}(x)) \ \zeta_j(h(x)) \ dxdt
\\[5pt]
& \quad+ \int_0^T\int_{S} u(t,x)A(x)\nabla \mathcal{K}u(t,x) \cdot \nabla h(x) \ \zeta_j'(h(x)) \ \gamma(t,\Psi_{h(x)}^{-1}(x)) \ dxdt.
\end{aligned}
$$
Consequently, applying the Coarea Formula for the function $h$, we obtain
 $$
 \begin{aligned}
 &\int^1_0\zeta_j(\tau)\int^T_0\int_{\partial\Omega_\tau}u(t,r)\dfrac{\partial_t\gamma(t,\Psi_{\tau}^{-1}(r))}{|\nabla h(r)|}\,d\clg{H}^{n-1}dtd\tau
 \\[5pt]
  &=\int^1_0\zeta_j(\tau)\int^T_0\int_{\partial\Omega_\tau}u(t,r)A(r)\nabla\mathcal{K}u(t,r)\cdot
  \dfrac{\nabla\gamma(t,\Psi_{h(\cdot)}^{-1}(\cdot))(r)}{|\nabla h(r)|}\,d\clg{H}^{n-1}(r)dtd\tau
 \\[5pt]
 &+\int^1_0\zeta_j'(\tau)\int^T_0\int_{\partial\Omega_\tau}
 u(t,r)A(r)\nabla\mathcal{K}u(t,r)\cdot\nu_\tau(r) \gamma(t,\Psi_{\tau}^{-1}(r))\,d\clg{H}^{n-1}(r)dtd\tau.
 \end{aligned}
 $$
Therefore, applying the Dominated Convergence Theorem in the above equation,
we get in the limit as $j \to \infty$
\begin{equation}
   J(\tau_2)+\int_0^{\tau_2}\Phi(\tau)d\tau=J(\tau_1)+\int_0^{\tau_1}\Phi(\tau)d\tau,\label{eq:J(tau)}
\end{equation}
 for all $\tau_1,\,\tau_2\in F$ and $\gamma\in\clg{F}$, where $\Phi(\tau)$ is given by
 $$
 \int^T_0\int_{\partial\Omega_\tau}u(t,r)\left(\dfrac{\partial_t \gamma(t,\Psi_{\tau}^{-1}(r))}{|\nabla h(r)|}
 -A(r)\nabla\mathcal{K}u(t,r)\cdot\dfrac{\nabla\gamma(t,\Psi_{h(\cdot)}^{-1}(\cdot))(r)}{|\nabla h(r)|}\right)\,d\clg{H}^{n-1}(r)dt.
 $$
 
 \smallskip
On the other hand, since $\clg{F}$ is dense in $C^{\infty}_c((0,T)\times\R^n)$, we have that \eqref{eq:J(tau)}
holds for $\gamma\in C^{\infty}_c((0,T)\times\R^n)$. Consequently, we obtain 
\begin{equation}
    \lim_{F\ni\tau\,\to\,0^+}\int^T_0\int_{\partial\Omega_{\tau}}u(t,r)A(r)\nabla\mathcal{K}u(t,r)
    \cdot\nu_\tau(r) \gamma(t,\Psi^{-1}_\tau(r))\,d\clg{H}^{n-1}(r)dt\nonumber
\end{equation}
exists for all  $\gamma\in C^{\infty}_c((0,T)\times\R^n)$.

%\medskip
%Now, using the same technique of the Lemma \ref{equiboundary} and without loss of generality, we obtain from \eqref{eq:limitesen} that
%\begin{equation}
%\begin{aligned}   
%    &\ess\!\!\lim_{\tau \to 0^+} \int_0^{T}
%    \!\!\!\int_{\partial\Omega_\tau} u(t,r)A(r)\nabla\mathcal{K}u(t,r)\cdot\nu_\tau(r)\gamma(t,\Psi^{-1}_\tau(r))\,d\clg{H}^{n-1}(r)dt=
%    \\[5pt]
%&\ess\!\!\lim_{\tau \to 0^+}\int^T_0
%\!\!\!\int_{\partial\Omega} u(t,\Psi_{\tau}(r))A(r)\nabla \mathcal{K}u(t,\Psi_{\tau}(r))\cdot\nu_\tau(\Psi_{\tau}(r))J[\Psi_\tau]\gamma(t,r) \ d\clg{H}^{n-1}(r)dt,
%\end{aligned}\nonumber
%\end{equation}
%where $J[\Psi_\tau]$ is defined in \eqref{eq:Jpsi}. Therefore, we have
%\begin{equation}
%    u(t,\Psi_{\tau}(r))A(r)\nabla \mathcal{K}u(t,\Psi_{\tau}(r))\cdot\nu_\tau(\Psi_{\tau}(r))J[\Psi_\tau](r),\label{eq:limit}
%\end{equation}
%is uniformly bounded in $L^2((0,T)\times \Gamma)$, with respect to $\tau>0$ small enough.
%Thus,  in view of the density of $C^{\infty}_c((0,T)\times\R^n)$ in $L^2((0,T)\times\partial\Omega)$ 
%and together with \eqref{eq:limit}, we obtain that \eqref{eq:limitesen} follows for all  
%$\gamma \in L^2((0,T)\times\partial\Omega)$. Finally, due to Lemma \ref{equiboundary},
%we obtain \eqref{existboundary}.
\end{proof}
The following result expresses in convenient
way the concept of (weak) solution of the IBVP \eqref{FTPME} as given by
Definition \ref{DEFSOL}.

\begin{theorem}[Equivalence Theorem]
\label{THEQUIVA}
A function $$u \in L^2\left((0,T); D\big(\mathcal{L}^{(1-s)/2}\big)\right) \cap L^\infty(\Omega_T)$$ 
is a weak solution of  the IBVP \eqref{FTPME} if, and only if, it satisfies 
\begin{equation}
    \iint_{\Omega_T} u(t,x) \ (\partial_t\phi-A(x)\nabla \mathcal{K}u\cdot \nabla \phi)dxdt
    +\int_{\Omega} u_0(x) \, \phi(0,x) \ dx= 0,\label{eq:thequiva}
\end{equation}
for each test function $\phi \in C^{\infty}_c((-\infty,T) \times \R^n)$. 
\end{theorem}
\begin{proof}
1. Assume that $u$ satisfies \eqref{eq:thequiva}, then we show that $u$ verifies \eqref{eq:solfrac}--\eqref{eq:boundcondition}.
To show \eqref{eq:solfrac}, it is enough to consider test functions $\phi \in C^{\infty}_c(\Omega_T)$.
In order to show \eqref{eq:initalcondition}, let us consider $\phi(t,x)= \gamma_j(t)\zeta(x)$, 
$\gamma_j(t)=H_j(t+t_0)-H_j(t-t_0)$ for any $t_0 \in E$ (fixed), and 
$\zeta \in \clg{E}$. 
Then, from \eqref{eq:thequiva} we have
$$
\begin{aligned}
      \iint_{\Omega_T} u(t,x) \gamma_j'(t)\zeta(x) \ dxdt
               &+\int_{\Omega} u_0(x) \zeta(x) \ dx
\\[5pt]               
               &=\iint_{\Omega_T} u(t,x) A(x)\nabla\mathcal{K}u(t,x) \cdot \nabla\zeta(x) \gamma_j(t) \ dxdt.
\end{aligned}
$$
Passing to the limit in the above equation as $j \to \infty$, and taking into account that $t_0$ is Lebesque point of $I(t)$, we obtain
%and also that $\gamma_j(t)$ converge pointwise to the characteristic function of the interval $[-t_0,t_0)$, we obtain
\begin{equation}
\begin{aligned}
   \int_{\Omega} u(t_0,x)\zeta(x) dx  =&\int_{\Omega} u_0(x)\zeta(x)dx 
\\[5pt]   
&\quad-\int^{t_0}_0\int_{\Omega}u(x)A(x)\nabla\mathcal{K}u(x)\cdot\nabla\zeta(x)dxdt,
\end{aligned}\label{INT100}
\end{equation} 
where we have used the Dominated Convergence Theorem. Since $t_0 \in E$ is arbitrary, and in view of the density of $\clg{E}$ 
in $L^1(\Omega)$, it follows from \eqref{INT100} that 
$$
    \ess\lim_{t\to 0}\int_{\Omega} u(t,x)\zeta(x) dxdt  =\int_{\Omega} u_0(x) \zeta(x) \ dx
$$
for all $\zeta \in L^1(\Omega)$, which shows \eqref{eq:initalcondition}.

\medskip
Finally, let us show \eqref{eq:boundcondition}. Similarly to   proof in Proposition \ref{PROP10} (2), we choose
$$
 \phi(t,x)=\left\lbrace
 \begin{aligned}
 \gamma(t,\Psi_{h(x)}^{-1}(x))\zeta_j(h(x)), &\quad \text{for $x \in S$}, 
 \\[5pt]
0\;, & \quad \text{for $x \in \Omega \setminus S$},
 \end{aligned}\right.
$$
where $\gamma \in \clg{F}$, $\zeta_j(\tau)=H_j(\tau+\tau_0)-H_j(\tau-\tau_0)$, with $\tau_0\in F$,
and $S= \Psi(F \times \partial \Omega)$.
Therefore, from \eqref{eq:thequiva} we obtain
$$
\begin{aligned}
\int_0^T\int_{S} &u(t,x) \ \partial_t \gamma(t,\Psi_{h(x)}^{-1}(x)) \ \zeta_j(h(x)) \ dxdt
\\[5pt]
&=\int_0^T\int_{S}u(t,x) \ A(x) \nabla \mathcal{K}u(t,x) \cdot \nabla\gamma(t,\Psi_{h(x)}^{-1}(x))\zeta_j(h(x)) \ dxdt
\\[5pt]
&+ \int_0^T\int_{S}u(t,x) \ A(x) \nabla \mathcal{K}u(t,x) \cdot \nabla h(x)\zeta_j'(h(x))\gamma(t,\Psi_{h(x)}^{-1}(x)) \ dxdt.
\end{aligned}
$$
On other hand, applying the Coarea Formula for the function $h$ in the above equation, we have
$$
 \begin{aligned}
 &\int^1_0\zeta_j(\tau)\int^T_0\int_{\partial\Omega_\tau}u(t,r)\dfrac{\partial_t\gamma(t,\Psi_{\tau}^{-1}(r))}{|\nabla h(r)|}\,d\clg{H}^{n-1}(r)dtd\tau
 \\[5pt]
  &=\int^1_0\zeta_j(\tau)\int^T_0\int_{\partial\Omega_\tau}u(t,r) \ A(r)\nabla\mathcal{K}u(t,r)
  \cdot \dfrac{\nabla\gamma(t,\Psi_{h(x)}^{-1}(x))(r)}{|\nabla h(r)|}\,d\clg{H}^{n-1}(r)dtd\tau
 \\[5pt]
 &+\int^1_0\zeta_j'(\tau)\int^T_0\int_{\partial\Omega_\tau}u(t,r) \ A(r)\nabla\mathcal{K}u(t,r)
 \cdot \nu_\tau(r) \gamma(t,\Psi_{\tau}^{-1}(r))\,d\clg{H}^{n-1}(r)dtd\tau.
 \end{aligned}
$$
Then, passing to the limit in the above equation as 
$j\to\infty$ and taking into account that $\tau_0$ is a
Lebesque point of $J(\tau)$, and also that $\zeta_j(t)$ 
converges pointwise to the characteristic function of the interval $[-\tau_0,\tau_0)$, we obtain
\begin{equation}
   J(\tau_0)=\int_0^{\tau_0}\Phi(\tau)d\tau,\label{eq:J(tau)}
\end{equation}
 for all $\tau_0\in F$ and $\gamma\in\clg{F}$, where $\Phi(\tau)$ is given by
 $$
 \int^T_0\int_{\partial\Omega_\tau}u(t,r)\left(\dfrac{\partial_t \gamma(t,\Psi_{\tau}^{-1}(r))}{|\nabla h(r)|}
 -A(r)\nabla\mathcal{K}u(t,r)\cdot\dfrac{\nabla\gamma(t,\Psi_{h(\cdot)}^{-1}(\cdot))(r)}{|\nabla h(r)|}\right)\,d\clg{H}^{n-1}(r)dt.
 $$
 
 \smallskip
On the other hand, since $\clg{F}$ is dense in $C^{\infty}_c((0,T)\times\R^n)$, we have that \eqref{eq:J(tau)}
holds for $\gamma\in C^{\infty}_c((0,T)\times\R^n)$. Then, for each $\tau\in F$ we have
$$
\left| J(\tau)\right|\leq C\left| \Psi((0,\tau)\times\partial\Omega)\right|,
$$
where $C$ is a positive 
constant, which does not depend on $\tau$. 
Hence passing to the limit as $\tau \to 0$, we obtain
$$
\lim_{F\ni\tau\,\to\,0}\int^T_0 \!\!\! \int_{\partial\Omega_{\tau}}u(t,r) \ A(r)\nabla\mathcal{K}u(t,r)
\cdot\nu_\tau(r) \gamma(t,\Psi^{-1}_\tau(r))\,d\clg{H}^{n-1}(r)dt= 0,
$$
for all $\gamma \in C^{\infty}_c((0,T)\times\R^n)$.

\medskip
2. Now, let us consider: \eqref{eq:solfrac}--\eqref{eq:boundcondition} $\Rightarrow$ \eqref{eq:thequiva}. 
The idea is similar to that one done before; for completeness we give the main points. 
First, we consider $j\in\Bbb{N}$ sufficiently large and take for any $t_0 \in E$
$$
  \phi(t,x)=\psi(t,x)H_j(t-t_0),
$$
where $\psi\in C_c^{\infty}((-\infty,T)\times\Omega))$, $H_j(t)$ as considered before. Then, 
from \eqref{eq:solfrac} we obtain
$$
\begin{aligned}
  &\iint_{\Omega_T}u(t,x)\partial_t\psi(t,x)H_j(t-t_0) \ dxdt+\iint_{\Omega_T} u(t,x) \ \psi(t,x) \ H'_j(t-t_0) \ dxdt
\\[5pt]  
  &-\iint_{\Omega_T} u(t,x) A(x)\nabla\mathcal{K}u(t,x) \cdot \nabla\psi(t,x) H_j(t-t_0) \ dxdt=0.
\end{aligned}
$$
Passing to the limit as $j\to \infty$, and taking into account that $t_0$ is a 
Lebesgue point of $I(t)$, also that $H_j(\cdot-t_0)$ converges pointwise to the Heaviside function $H(\cdot-t_0)$, 
after that, taking the limit as $E\ni t_0\to 0$ and using \eqref{eq:initalcondition}, we have
\begin{equation}
\label{eq:infTxOmega}
\begin{aligned}
  \iint_{\Omega_T} u(t,x) \partial_t \psi(t,x)&\ dxdt +\int_{\Omega}u_0(x) \ \psi(0,x) \ dx
  \\[5pt]
  & - \iint_{\Omega_T} u(t,x)A(x) \nabla\mathcal{K}u(t,x) \cdot \nabla\psi(t,x) \ dxdt=0,
\end{aligned}
\end{equation}
for all $\psi\in C_c^{\infty}((-\infty,T)\times\Omega)$. In particular, for
$$
  \psi(t,x)=\phi(t,x)(1-\zeta_j(h(x)),
$$
where $\phi\in C_c^{\infty}((-\infty,T)\times\Bbb{R}^n)$, $h(x)$ as above and we consider the function 
$\zeta_j(\tau)= H_j(\tau+\tau_0) - H_j(\tau-\tau_0)$, where $\tau_0\in F$. Then, from (\ref{eq:infTxOmega}) we obtain
$$
\begin{aligned}
\iint_{\Omega_T} u(t,x) \partial_t&\phi(t,x)(1-\zeta_j(h(x))) \ dxdt + \int_{\Omega} u_0(x) \ \phi(0,x)(1-\zeta_j(h(x))) \ dx
\\[5pt]
&- \iint_{\Omega_T} u(t,x) A(x)\nabla \mathcal{K}u(t,x) \cdot \nabla\phi(t,x) \ (1-\zeta_j(h(x))) \ dxdt
\\[5pt]
&+  \iint_{\Omega_T}  u(t,x) A(x)\nabla\mathcal{K}u(t,x) \cdot\nabla h(x) \ \zeta'_j(h(x)) \phi(t,x) \ dxdt= 0.
\end{aligned}
$$
Finally, we use the Coarea Formula for the function $h$ in the last integral of the above equation, and
pass to limit as $j \to \infty$. Therefore, we obtain for all $\phi\in C_c^{\infty}((-\infty,T)\times\R^n)$
$$
    \iint_{\Omega_T}\!u(t,x)(\partial_t\phi(t,x) - A(x)\nabla\mathcal{K}u(t,x) \cdot \nabla\phi(t,x) )\, dxdt
    +\int_{\Omega}u_0(x)\phi(0,x)dxdt=0,
$$
where we have used \eqref{eq:boundcondition}.
\end{proof}
%%%%%%%%%%%%%%%%%
%%%%%%%%%%%%%%%%%%%%%%%%%%%%%%%%%%
\subsection{Solution estimates for the IBVP}
\label{Basic estimates}
%%%%%%%%%%%%%%%%%

Now, we show basic estimates, which are required 
to show existence of weak solutions
to the IBVP \eqref{FTPME}. We perform formal calculations,
assuming that $u\geq0$ satisfies the required smoothness and integrability assumptions.
 
\begin{enumerate}
\item \emph{Conservation of mass}: For all $t\in(0,T)$,
$$
\dfrac{d}{dt}\int_{\Omega} u(t,x) \ dx= \int_{\Omega}\dive(uA(x)\nabla\mathcal{K}u) \ dx=0.
$$
 
\item \emph{Conservation of positivity}: If  the initial condition $u_0$ is non-negative, then the solution $u$ of (\ref{FTPME}) is non-negative.

\medskip
Indeed, we assume $u_0> 0$ (without loss of generality). 
For any $0< t_0 \leq T$ (fixed), let $x_0\in\Omega$ be a point where $u(t_0,\cdot)$ is a minimum, which is to say
$$
   u(t_0,x)\geq u(t_0,x_0) \quad \text{for each $x \in \Omega$}.
$$
We claim that $u(t_0,x_0)\geq0$. 
%for any $0< t_0 \leq T$. 
Note that, since $t_0$ is arbitrary, this sentence implies that $u$ is non-negative. 
Let us suppose that, $u(t_0,x_0)<0$,
%Certainly, assume that exist $t_0\in(0,T]$ and $x_0\in\Omega$ such that $u(t_0,x_0)<0$.
and consider for each $\delta>0$,
%Our proof starts with the construction of $\varphi_\delta$. For each $\delta>0$, we consider
\begin{equation}
\varphi_{\delta}(w)= \left\{
\begin{aligned}
 (w^2+\delta^2)^{1/2}-\delta, &\quad \text{for $0 \leq w$},
 \\[5pt]
0 \quad,& \quad \text{for $w \leq 0$}.
\end{aligned}
\right.
\end{equation}
Then, $\varphi_\delta(w)$ converges to $w^+=\max\{w,0\}$ as $\delta\to0^+$. Now, multiplying 
the first equation in \eqref{FTPME} by $\varphi'_{\delta}(u)$ and evaluating in $(t_0,x_0)$, we obtain
$$
\begin{aligned}
    \frac{d}{dt} \varphi_{\delta}(u)(t_0,x_0)&= \nabla u(t_0,x_0) \cdot A(x_0)\nabla\mathcal{K}u(t_0,x_0) \ \varphi'_{\delta}(u(t_0,x_0))
    \\[5pt]
    &\;+ u(t_0,x_0) \ \varphi'_{\delta}(u(t_0,x_0)) \ \dive(A(\cdot)\nabla\mathcal{K}u)(t_0,x_0).
    \end{aligned}
$$
The first term in the right hand side of the above equation
is zero, since $x_0$ is a point where $u(t_0,\cdot)$ is a minimum.
For the second term, we recall that $-\dive(A(\cdot)\nabla\mathcal{K}u)=\mathcal{L}(\mathcal{K}u)=\mathcal{L}^{1-s}u$, hence due to Lemma \ref{lem:semihet},
it follows that
\begin{equation}
-\dive(A(\cdot)\nabla\mathcal{K}u)(t_0,x_0)=\dfrac{1}{\Gamma(s-1)}\int^{\infty}_0
\big( e^{-t\mathcal{L}}u(t_0,x_0)-u(t_0,x_0) \big) \dfrac{dt}{t^{2-s}},
\label{eq:DeltaK}
\end{equation}
where $\Gamma(s-1)< 0$ $(s< 1)$ and $v(t,x)= e^{-t\mathcal{L}}u(t_0,x)$ is the weak solution of the IBVP
$$
\left\lbrace\begin{aligned}
   \partial_t v+\mathcal{L} v&=0,\;\;\;\;\;\;\;\;\;\mbox{in}\;(0,T)\times\Omega,
   \\[3pt]
   v(t,x)&= 0,\;\;\;\;\;\;\;\;\;\mbox{on}\;[0,T]\times\partial\Omega,
   \\[3pt]
   v(0,x)&=u(t_0,x),\mbox{in}\;\Omega.
\end{aligned}\right.
$$
Now, applying the (weak) maximum principle, we get
\begin{equation}
\min_{(t,x)\,\in\,\bar{\Omega}_T} e^{-t\mathcal{L}}u(t_0,x)= \min_{(t,x)\,\in\,\Gamma_T} e^{-t\mathcal{L}}u(t_0,x),
\label{eq:heatforv}
\end{equation}

where $\Gamma_T$ is the parabolic boundary of $\Omega_T$, which comprises 
$\left\lbrace0\right\rbrace\times\Omega$ and 
$[0,T]\times\partial\Omega$. Consequently, we have from \eqref{eq:heatforv}
$$
   \min_{(t,x)\,\in\,\bar{\Omega}_T} e^{-t\mathcal{L}}u(t_0,x)=
   \min\left\lbrace 0,\min_{x\in\,\Omega}u(t_0,x)\right\rbrace.
%   =\min\left\lbrace 0,u(t_0,x_0)\right\rbrace,
$$
%where the last equality is because $x_0$ 
%is a point that minimizes the function $u(t_0,\cdot)$. 
%Since $\min\left\lbrace 0,u(t_0,x_0)\right\rbrace= u(t_0,x_0)$
Therefore, it follows that
$e^{-t\mathcal{L}}u(t_0,x)\geq u(t_0,x_0)$,  
for all $x\in\Omega$. Thus from \eqref{eq:DeltaK} we deduce that, 
$-\dive(A(\cdot)\nabla\mathcal{K}u)(t_0,x_0) \geq 0$. Moreover, since 
$u \ \varphi'_{\delta}(u) \geq 0$, we have at $(t_0,x_0)$ that $\frac{d}{dt}\varphi_{\delta}(u)\geq 0$, and thus
\begin{equation}
\label{eq:varut}
     \varphi_{\delta}(u(t_0))\geq \varphi_{\delta}(u_0).
\end{equation}
Then, passing to the limit in \eqref{eq:varut} as $\delta \to 0$, we obtain $u^{+}(t_0) \geq u_0$,
which implies that $u(t_0,x_0)> 0$, which is a contradiction, hence $u$ is non-negative.

\item  \emph{$L^{\infty}$ estimate}: The $L^{\infty}$ norm of $u$ does not increase in time.

\medskip
Indeed, for any $0< t_0 \leq T$ (fixed), let $x_0$ be a point where $u(t_0,\cdot)$ is a maximum, which is to say
$$
   u(t_0,x)\leq u(t_0,x_0) \quad \text{for all $x \in \Omega$}.
$$
Therefore, we have
$$
    \frac{du}{dt}(t_0,x_0) = \nabla u(t_0,x_0) \cdot A(x_0) \nabla\mathcal{K}u(t_0,x_0)
   + u(t_0,x_0)\dive(A(\cdot)\nabla\mathcal{K}u)(t_0,x_0).
 $$
The first term in the right hand side of the above equation
is zero, since $x_0$ is a point where $u(t_0,\cdot)$ is a maximum.
For the second term, we use the same ideas as above, thus $\dive(A(\cdot)\nabla\mathcal{K}u)(t_0,x_0) \leq 0$. 
Moreover, since $u\geq0$, then at $(t_0,x_0)$ we have $\frac{du}{dt}\leq 0$, which implies item 3.

\item  \emph{First energy estimate}: For all $t\in(0,T)$,
$$
\int_{\Omega}u(t,x) \log u(t,x) \ dx  +\Lambda_1 \int_0^t\int_{\Omega} \vert \nabla \mathcal{H}u(t',x) \vert^2 \ dxdt'\leq   \int_{\Omega} u_0(x) \log u_0(x) \ dx
$$

Indeed, multiplying the first equation \eqref{FTPME} by $\log u(t',x)$ and integrate on $\Omega$. Then after integration by part, we obtain
$$
\dfrac{\partial}{\partial t}\int_{\Omega} u(t',x)\log u(t',x)dx+
\int_{\Omega}A(x)\nabla \mathcal{K}u(t',x)\cdot \nabla u(t',x)dx=0.
$$
On the other hand, from Proposition \ref{proKuHu} (2), we have
$$
\dfrac{\partial}{\partial t}\int_{\Omega} u(t',x)\log u(t',x)dx+
\Lambda_1\int_{\Omega}|\nabla \mathcal{H}u(t',x)|^2dx 
\leq0
$$ 
Then, we integrate over $(0,t)$, for all $0<t<T$, to obtain the first energy estimate. 
\item  \emph{Second energy estimate}: Similar to the above description, is not difficult to show that
$$
   \dfrac{1}{2}\int_{\Omega} \vert \mathcal{H} u(t_2,x) \vert^2 dx'
   +\Lambda_1\int_{t_1}^{t_2}\int_{\Omega} u(t',x) \vert \nabla \mathcal{K}u(t',x) \vert^2 dxdt\leq \dfrac{1}{2}\int_{\Omega}\vert \mathcal{H}u(t_1,x)\vert^2 dx,
$$
for $0 \leq t_1 < t_2  \leq T$.
\end{enumerate}

%%%%%%%%%%%%%%%%%%%%%
\section{Existence of Weak Solutions}
%%%%%%%%%%%%%%%%%%%%%
The aim of this section is to find a weak solution of \eqref{FTPME}. To show that we use the equivalent definition given by the theorem \ref{THEQUIVA}. The following theorem show the existence of weak solution.

\begin{theorem}
\label{Thprincipal}
Let $u_0 \in L^{\infty}(\Omega)$
be a non-negative function. 
Then, there exists a weak solution $u \in L^2\big((0,T);D\big(\mathcal{L}^{(1-s)/2}\big)\big)\cap L^{\infty}(\Omega_T)$ 
of the IBVP \eqref{FTPME}.
\end{theorem}
The proof will be divided into two subsections.

%%%%%%%%%%%%%%%%%%%
\subsection{Smooth Solution}
\label{Parabolic regularization}
%%%%%%%%%%%%%%%%%%%
To show the existence of the solution we use the method of vanishing
viscosity and also it will be eliminated the degeneracy by raising the 
level set $\{u = 0\}$ in the diffusion coefficient. The basic idea of which is as follows: for  $\delta,\;\mu\in (0, 1)$ we study
the parabolic perturbation of the Cauchy problem \eqref{FTPME} given by
\begin{eqnarray}
     \partial_t u_{\mu,\delta}
     -\delta\dive(A(x)\nabla  u_{\mu,\delta})&=&
     \dive(q(u_{\mu,\delta}) A(x) \nabla \mathcal{K} u_{\mu,\delta}) 
     \;\hspace{0,6cm}\;\mbox{in}\;\Omega_T,\label{eq:aproxequation1}\\
u_{\mu,\delta}(0,\cdot)&=&u_{0\delta} \hspace{3cm}\mbox{in}\;\Omega,\label{eq:aproxequation2}\\
 u_{\mu,\delta}&=& 0 \hspace{3,2cm}\mbox{on}\; (0,T) \times \partial \Omega,\label{eq:aproxequation3}
\end{eqnarray}
where $q(\lambda)= \lambda+\mu$, and $u_{0\delta}$ is a non-negative
smooth bounded approximation of the initial data $u_0\geq 0$, 
satisfying $u_{0\delta}= 0$ on $\partial \Omega$.
% on $\overline{\Gamma_T}\cap \overline{\Omega}$. 

\medskip
Now,  we make use of the well known results of existence, uniqueness and uniform $L^{\infty}$ bounds 
for quasilinear parabolic problems. Therefore, for each $\delta,\;\mu> 0$, 
there exists a unique classical solution 
$u_{\mu,\delta} \in C^2(\Omega_T) \cap C\big(\bar{\Omega}_T\big)$
of the IBVP \eqref{eq:aproxequation1}--\eqref{eq:aproxequation3}, (see \cite{LSU68}, p. 449). 

\medskip
The following theorem investigates the properties of the solution
$u_{\mu,\delta}$ to the parabolic perturbation \eqref{eq:aproxequation1}--\eqref{eq:aproxequation3}  for fixed $\delta,\;\mu\in (0, 1)$.
\begin{theorem}
\label{Theaprox}
For each $\mu,\delta>0$, let $u= u_{\mu,\delta}  
\in C^2(\Omega_T) \cap C\big(\bar{\Omega}_T\big)$ be the unique classical solution of  
\eqref{eq:aproxequation1}--\eqref{eq:aproxequation3}. Then, $u$ satisfies:
\begin{enumerate}
\item[$(1)$] For all $\phi\in C_c^{\infty}((-\infty,T)\times\Bbb{R}^n)$,
\begin{equation}
\label{eq:thequiaproxsol}
\hspace{-0,57cm}\begin{aligned}
    \iint_{\Omega_T} u(t,x)(\partial_t \phi(t,&x)-\delta\mathcal{L}\phi(t,x)) \ dxdt 
    + \int_{\Omega}u_{0\delta}(x) \ \phi(0,x) \ dx
 \\[5pt]
    &=\iint_{\Omega_T} q(u(t,x)) A(x) \nabla \mathcal{K}u(t,x) \cdot \nabla\phi(t,x) \ dxdt.
\end{aligned}
\end{equation}

\item[$(2)$] For each $t \in (0,T)$, we have 
\begin{equation}
\label{eq:+ebounded}
   \Vert u(t)\Vert_{\infty}\leq\Vert u_0\Vert_{\infty},
\end{equation}  
and the conservation of the ``total mass"
\begin{equation}
\label{eq:consermassaprox}
\int_{\Omega}u(t,x) \ dx = \int_{\Omega} u_{0\delta}(x) \ dx 
\leq \int_{\Omega} u_0(x) \ dx.
\end{equation}
Furthermore, for all $(t,x)\in\Omega_T$, $0 \leq u(t,x)$.

\item[$(3)$] First energy estimate: For $\eta(\lambda):= (\lambda+\mu) \log (1+(\lambda/\mu))-\lambda$, $(\lambda\geq 0)$, and all $t\in (0,T)$,
\begin{equation}
\begin{aligned}
\int_{\Omega}\eta(u(t))\,dx+&\Lambda_1\delta\int_0^t\int_{\Omega}\dfrac{\vert \nabla u \vert^2}{q(u)}\,dx\,dt\,
\\[5pt]
& +\Lambda_1\int_0^t\int_{\Omega}|\nabla\mathcal{H} u|^2\,dxdt\leq   \int_{\Omega}\eta(u_{0\delta})dx
\label{eq:Fu}
\end{aligned}
\end{equation}

\item[$(4)$] The second energy estimate: For all $0 < t_1 < t_2 < T$, 
\begin{equation}
\begin{aligned} 
\frac{1}{2}\int_{\Omega}&|\mathcal{H} u(t_2,x)|^2\,dx+\Lambda_1\delta \int_{t_1}^{t_2}\int_{\Omega}|\nabla\mathcal{H} u|^2\,dxdt
 \\[5pt]
&+\Lambda_1\int_{t_1}^{t_2}\int_{\Omega} q(u) \left|\nabla \mathcal{K} u\right|^2\,dxdt
  \leq \frac{1}{2}\int_{\Omega}|\mathcal{H} u(t_1,x)|^2 \ dx
\end{aligned}
\label{eq:2energystimate}
\end{equation}

\item[$(5)$] For all $v\in H^1_0(\Omega)$,
\begin{equation}
\begin{aligned}
\hspace{-0,56cm}\int^T_0 \langle \partial_t u(t), v\rangle dt=-\delta&\iint_{\Omega_T}A(x)\nabla u\cdot\nabla v\,dxdt 
\\[5pt]
&+ \iint_{\Omega_T} q(u)A(x)\nabla \mathcal{K}u\cdot \nabla v \,dxdt.
\end{aligned}
\label{eq:L2Hs-1}
\end{equation}
where $\langle\cdot,\cdot\rangle$ denote the pairing between $H^{-1}(\Omega)$ and $H^1_0(\Omega)$.
\end{enumerate}
\end{theorem}
\smallskip
\begin{proof}
The first part of the theorem (up to \eqref{eq:consermassaprox}) is analogous to the theorem 4.2 \cite{Hua} and therefore we omit the proofs. We will show \eqref{eq:Fu}- \eqref{eq:L2Hs-1}.

\medskip
(1) To get the first energy estimate (\ref{eq:Fu}), we multiply 
equation \eqref{eq:aproxequation1}
by $\eta'(u)$ and integrate on $\Omega$. Then, after integration by parts and taking into account that $\eta'(0)=0$, we have
$$
   \dfrac{\partial}{\partial t}\int_{\Omega}\eta(u)dx=-\delta\int_{\Omega}
   \dfrac{1}{q(u)} A(x)\nabla u\cdot\nabla u  \ dx
   - \int_{\Omega}A(x)\nabla\mathcal{K}u\cdot\nabla u\,dx.
$$

Then, we integrate over $(0,t)$, for all $0 < t < T$, to obtain 
$$
   \int_{\Omega}\eta(u(t))dx+\delta\int_0^t\int_{\Omega}
   \dfrac{1}{q(u)} A(x)\nabla u\cdot\nabla u  \ dx
   + \int_0^t\int_{\Omega}A(x)\nabla\mathcal{K}u\cdot\nabla u\,dx=\int_{\Omega}\eta(u_0)dx.
$$

On the other hand, due to the uniform ellipticity condition we have an estimate for the second term of the left hand side 
$$
\Lambda_1\int_0^t\int_{\Omega}\dfrac{|\nabla u |^2}{q(u)}  \ dx\leq\int_0^t\int_{\Omega}\dfrac{1}{q(u)} A(x)\nabla u\cdot\nabla u  \ dx
%\leq\Lambda_2 
%\int_0^t\int_{\Omega}\dfrac{|\nabla u |^2}{q(u)}  \ dx,
$$
and for the third term of the left hand side, we use proposition \ref{proKuHu} item (2), thus we obtain \eqref{eq:Fu}.

\medskip
(2) To prove \eqref{eq:2energystimate}, we multiply \eqref{eq:aproxequation1} by $\xi_k\mathcal{K}u$,
integrate over $\Omega$ and take into account that $\xi_k= 0$ on $\partial\Omega$. Then, we have
$$
\int_{\Omega}\xi_k\frac{\partial u}{\partial t}\mathcal{K}u\,dx+\delta\int_{\Omega}A(x)\nabla u\cdot\nabla(\xi_k\mathcal{K}u)\,dx 
+\int_{\Omega} q(u)A(x)\nabla \mathcal{K} u\cdot\nabla (\xi_k\mathcal{K}u)\,dx=0.
$$

Passing to the limit as $k \to \infty$ and using Lemma \ref{Lemma:xik}, it follows that
$$
\frac{1}{2}\frac{\partial}{\partial t}\int_{\Omega}|\mathcal{H}u(t)|^2dx
+\delta\int_{\Omega}A(x)\nabla u\cdot\nabla\mathcal{K}u\,dx 
+\int_{\Omega} q(u)A(x)\nabla \mathcal{K} u\cdot\nabla \mathcal{K}u\,dx=0.
$$

Then, integrating over $(t_1,t_2)$  we get
$$
\begin{aligned}
\frac{1}{2}\int_{\Omega}|\mathcal{H}u(t_2,x)|^2dx.+\delta\int_{t_1}^{t_2}\int_{\Omega}A(x)\nabla u\cdot\nabla\mathcal{K}u\,dx 
&+
\\[5pt]
\int_{t_1}^{t_2}\int_{\Omega} q(u)A(x)\nabla \mathcal{K} u\cdot\nabla \mathcal{K}u\,dx&=
\frac{1}{2}\int_{\Omega}|\mathcal{H}u(t_1,x)|^2dx.
\end{aligned}
$$

On the other hand, from the uniform ellipticity condition we have and estimate for the third term of the left hand side 
$$
\Lambda_1 
\int_{t_1}^{t_2}\int_{\Omega}q(u)|\nabla \mathcal{K}u |^2  \ dx\leq\int_{t_1}^{t_2}\int_{\Omega}q(u) A(x)\nabla \mathcal{K}u\cdot\nabla \mathcal{K} u  \ dx
%\leq\Lambda_2 
%\int_{t_1}^{t_2}\int_{\Omega}q(u)|\nabla \mathcal{K}u |^2  \ dx,
$$
and for the second term of the left hand side, we use the remark \ref{remarkAx}.
Therefore we get the second energy estimate \eqref{eq:2energystimate}, for all $0<t_1<t_2<T$.

\medskip
(3) It remains to show \eqref{eq:L2Hs-1}, which follows 
applying the same techniques above, so the proof is omitted. 
Hence the proof of the Theorem \ref{Theaprox} is complete.
\end{proof}

%%%%%%%%%%%%%%%%%%%%
\subsection{Limit transition}
%%%%%%%%%%%%%%%%%%%%
Here we pass to the limit in \eqref{eq:thequiaproxsol}, as the two parameters $\delta$, $\mu$ go to zero. To show that we use the first and the second energy estimates together with the Aubin-Lions' Theorem. After that we apply the Theorem \ref{THEQUIVA} to prove the existence of solution

\smallskip
As a first step, we define $u_\delta:= u_{\mu,\delta}$ (fixing $\mu> 0$). 
Then, we have the following 
\begin{proposition}
\label{Thprinc}
Let $\{u_\delta\}_{\delta>0}$ be the classical solutions of \eqref{eq:aproxequation1}--\eqref{eq:aproxequation3}. 
%Let $u_0\in L^{\infty}(\Omega)$, $u_0\geq 0$ and $\{u_\delta\}_{\delta>0}$ 
%is defined above. 
Then, there exists a subsequence of $\{u_\delta\}_{\delta>0}$, which 
weakly converges to some function $u\in  L^2\big((0,T); D\big(\mathcal{L}^{(1-s)/2}\big)\big) \cap L^{\infty}(\Omega_T)$,
satisfying
\begin{equation}
\begin{aligned}
\iint_{\Omega_T}u(t,x)\partial_t\varphi(t,x)&+\int_{\Omega}u_0(x)\varphi(0,x)dx
\\[5pt]
&=\iint_{\Omega_T} q(u(t,x)) A(x)\nabla\mathcal{K}u(t,x)\cdot\nabla \varphi(t,x) dxdt.
\end{aligned}
\label{eq:thequiaproxso13}
\end{equation}
For all $\varphi\in C^{\infty}_c((-\infty,T)\times\Bbb{R}^n)$
\end{proposition}

\begin{proof}
The idea of the proof of \eqref{eq:thequiaproxso13} is to pass to the limit in 
\eqref{eq:thequiaproxsol} as $\delta \to 0^+$. Therefore we need to show compactness of the sequence $\{u_\delta\}_{\delta>0}$. From \eqref{eq:+ebounded}, it follows that 
$\{u_\delta\}_{\delta>0}$ is (uniformly) bounded in $L^{\infty}(\Omega_T)$.
Then, it is possible to select a subsequence, still denoted by $\{u_\delta\}$,
converging weakly-$\star$ to $u$ in $L^{\infty}(\Omega_T)$, i.e.
$$
\lim_{\delta\to 0^+}\int_{\Omega_T}u_\delta(t,x)\phi(t,x)\,dtdx=\int_{\Omega_T}u(t,x)\phi(t,x)\,dtdx,
$$
for all $\phi\in L^1(\Omega_T)$, which is enough to pass to the limit in the first integral in the left hand side of \eqref{eq:thequiaproxsol}. 

\medskip
Now, we study the convergence of the integral in right hand side of \eqref{eq:thequiaproxsol}. First, since $A(x)$ is symmetric,  it is sufficient to show $q(u_\delta)\nabla\mathcal{K}u_\delta$ converges weakly in $\big(L^2(\Omega_T) \big)^n$.  
The proof will be divede into two step. First weak convergence of $\nabla\mathcal{K}u_\delta$ and strong convergence of $u_\delta$ in $\big(L^2(\Omega_T) \big)^n$.
From \eqref{eq:2energystimate}, we have 
$$
    \iint_{\Omega_T} |\nabla \mathcal{K} u_{\delta}|^2 \,dxdt
    \leq \dfrac{C}{\mu},
$$
where $C$ is a positive constant which does not depend on $\delta$.
Therefore, the right-hand side is (uniformly) bounded in $L^2(\Omega_T)$ 
w.r.t. $\delta$. Thus we obtain (along suitable subsequence) that,
$\nabla\mathcal{K}u_\delta$ converges weakly to $\textbf{v}$ in $(L^2(\Omega_T))^n$.

\medskip
The next step is to show that $\textbf{v}= \nabla\mathcal{K}u$ in $(L^2(\Omega_T))^n$. 
First we prove the regularity of $u$. From the equivalent norm \eqref{normequi}  we deduce that 
\begin{equation}
\begin{aligned}
   \iint_{\Omega_T} \left|\mathcal{L}^{(1-s)/2}u_\delta(t,x)\right|^2dxdt &
 \leq\Lambda_2\iint_{\Omega_T} \vert\nabla \mathcal{H}u_{\delta}\vert^2dxdt.
\end{aligned}\nonumber
\end{equation}
 
 On the other hand, from \eqref{eq:Fu}, we obtain that $\nabla\mathcal{H}u_{\delta}$ is (uniformly) bounded in $\big(L^2(\Omega_T)\big)^n$ w.r.t. $\delta$.
Thus $\{u_\delta\}$ is (uniformly) bounded in  $L^2\big((0,T); D\big(\mathcal{L}^{(1-s)/2}\big)\big)$. Consequently, it is possible to select a 
subsequence, still denoted by $\{u_\delta\}$, converging weakly to $u$ in $L^2\big((0,T); D\big(\mathcal{L}^{(1-s)/2}\big)\big)$, 
where we have used the uniqueness of the limit. Therefore, using again \eqref{normequi} and  the Poincare's type inequality  (corollary \ref{poincare}), follow that
$$
\iint_{\Omega_T} |\nabla\mathcal{K}u(t,x)|^2 dxdt\leq\Lambda_1^{-1}\lambda_1^{-s}\iint_{\Omega_T}|\mathcal{L}^{(1-s)/2}u(t,x)|^2dxdt,
$$
where $\lambda_1$ is the first eigenvalue of $\mathcal{L}$. Thus, we obtain that $\nabla\mathcal{K}u\in \big(L^2(\Omega_T)\big)^n$, and 
hence $\nabla\mathcal{K}u_\delta$ converges weakly to 
$\nabla\mathcal{K}u$  in $\big(L^2(\Omega_T)\big)^n$.

\medskip
Recall that, we are proving the weak convergence of $q(u_\delta)\nabla\mathcal{K}u_\delta$  in $\big(L^2(\Omega_T) \big)^n$. 
Now, we  prove strong convergence for $\{u_\delta\}_{\delta>0}$ in $L^2(\Omega_T)$, here we apply the Aubin-Lions compactness Theorem.
Indeed, from \eqref{eq:Fu}--\eqref{eq:L2Hs-1} and the (uniform) boundedness of $\nabla\mathcal{K}u_\delta$ in  $\big(L^2(\Omega_T)\big)^n$, we have
\begin{equation}
\int^T_0 \left\Vert \partial_t u_\delta\right\Vert^2_{H^{-1}(\Omega)} dt \leq C \ ( \Vert u_0 \Vert_{\infty} + \mu).
\label{eq:L2H-14.2}
\end{equation}
Observe that, at this point $\mu>$ is fixed. Thus, the right-hand side of \eqref{eq:L2H-14.2} is  bounded in $L^2((0,T);H^{-1}(\Omega))$ 
w.r.t. $\delta$. Therefore, exist a subsequence, such that $\partial_t u_\delta$ converges weakly to $\partial_t u$ in 
$L^2(0,T;H^{-1}(\Omega))$. Then, applying the Aubin-Lions compactness Theorem (see \cite{Malek}, Lemma 2.48)
it follows that, $u_\delta$ converges to $u$ (along suitable subsequence) strongly in 
$L^2(\Omega_T)$ as $\delta$ goes to zero. Consequently,
$q(u_\delta) \nabla\mathcal{K}u_\delta$ converges weakly to $q(u)\nabla\mathcal{K}u$ as $\delta\to 0^+$.
Hence, the equality \eqref{eq:thequiaproxso13} follows. 
\end{proof}
\begin{corollary}\label{corprinc}
Let $u$ the function given by the proposition \eqref{Thprinc}, satisfies:
\item[$(1)$] For almost all $t \in (0,T)$, 
\begin{equation}
\label{eq:bounded4.2}
     \Vert u(t)\Vert_{\infty} \leq  \Vert u_0 \Vert_{\infty}, \quad \text{and} 
\end{equation}
\begin{equation}
\label{eq:consevationofmass4.2}
    \int_{\Omega}u(x,t)dx = \int_{\Omega}u_0(x)dx.
\end{equation}
Furthermore, $0\leq u(t,x)$ a.e in $\Omega_T$.

\item[$(2)$] First energy estimate: For 
$\eta(\lambda):= (\lambda+\mu) \log (1+(\lambda/\mu))-\lambda$, $(\lambda\geq 0)$, and almost all $t\in (0,T)$, 
\begin{equation}
\int_{\Omega}\eta(u(t)) dx+\Lambda_1\int_0^t\int_{\Omega}\vert \nabla \mathcal{H}u\vert^2 \ dxdt'
\leq\int_{\Omega}\eta(u_0) \ dx.
\label{energyinequa4.2}
\end{equation} 

\item[$(3)$] Second energy estimate:  For almost all $0< t_1< t_2< T$,
\begin{equation}
\frac{1}{2}\int_{\Omega}\vert \mathcal{H} u(t_2) \vert^2 dx + \Lambda_1\int_{t_1}^{t_2}\int_{\Omega} q(u) 
\vert \nabla \mathcal{K}u \vert^2 \,dx\,dt\leq\frac{1}{2}\int_{\Omega}\vert \mathcal{H} u(t_1) \vert^2 dx.\label{eq:2energystimate4.2}
\end{equation}

\item[$(4)$] For each $v \in H^1_0(\Omega)$,
\begin{equation}
\int^T_0 \langle \partial_t u, v\rangle dt= \iint_{\Omega_T} q(u)A(x)\nabla \mathcal{K}u\cdot \nabla v \,dxdt.\label{estimateH-14.2}
\end{equation}
where $\langle\cdot,\cdot\rangle$ denote the pairing between $H^{-1}(\Omega)$ and $H^1_0(\Omega)$.
\end{corollary}
\begin{proof}
The proof of (\eqref{eq:bounded4.2} to \eqref{estimateH-14.2}) is standard, see \cite{Hua}, and therefore we omit the proofs.
\end{proof}
\begin{remark}
\label{remark4.2}
The function $u$ (obtained in the previous proposition) depends on the fixed parameter $\mu$.  
For each $\mu> 0$, we write from now on $u_\mu$ instead of $u$. 
\end{remark}

\begin{proof}[ \bf Proof of Theorem \ref{Thprincipal}]
To prove the existence of weak solution of the IBVP \eqref{FTPME}, we consider the sequence $\{u_\mu\}_{\mu>0}$, obtained in the proposition \ref{Thprinc}, which satisfies the corollary \ref{corprinc} for each $\mu> 0$, \eqref{eq:thequiaproxso13}--\eqref{estimateH-14.2}.

The idea of the proof is to pass to the limit in \eqref{eq:thequiaproxso13} as $\mu\to 0^+$, and obtain  
the solvability of the IBVP \eqref{FTPME}
applying the Equivalence Theorem \ref{THEQUIVA}. 

\medskip
From \eqref{eq:bounded4.2}, we see that 
$\{u_\mu\}_{\mu>0}$ is (uniformly) bounded in $L^{\infty}(\Omega_T)$ w.r.t $\mu$.
Hence, it is possible to select a subsequence, still denoted by $\{u_\mu\}$,
converging weakly-$\star$ to $u$ in $L^{\infty}(\Omega_T)$, 
%i.e.
%$$
%\lim_{\mu\to 0^+}\int_{\Omega_T}u_\mu(t,x)\phi(t,x)\,dtdx=\int_{\Omega_T}u(t,x)\phi(t,x)\,dtdx,
%$$
%for all $\phi\in L^1(\Omega_T)$, 
which is enough to pass to the limit in the first integral in the left hand side of \eqref{eq:thequiaproxso13}. 

\medskip
Now, we study the convergence of the integral in right hand side of \eqref{eq:thequiaproxso13}. First, since $A(x)$ is symmetric,  it is sufficient to show $q(u_\delta)\nabla\mathcal{K}u_\delta$ converges weakly in $\big(L^2(\Omega_T) \big)^n$. On the other hand, we recall that
$$
\begin{aligned}
\eta(\lambda)&= (\lambda+\mu)\log (1+\lambda/\mu)- \lambda, 
\\[5pt]
&=\lambda \log (\lambda+\mu)-\lambda\log \mu+\mu\log(1+\lambda/\mu)- \lambda, \qquad (\forall \lambda \geq 0)
\end{aligned}
$$
Then, from \eqref{eq:consevationofmass4.2} and \eqref{energyinequa4.2} we obtain for almost all $t\in (0,T)$
\begin{equation}
\begin{aligned}
    \Lambda_1\int_0^t\int_\Omega |\nabla \mathcal{H}u_\mu|^2 \ dxdt
    &+\int_\Omega u_\mu(t)\log(u_\mu(t)+\mu) \ dx
\\[5pt]
&\leq \int_{\Omega} u_0  \log(u_0+\mu) \ dx+\mu\int_{\Omega}  \log(1+u_0/\mu) \ dx,
\end{aligned}\label{eq:deslog}
\end{equation}
where we have used that $\mu\int_{\Omega}\log(1+u_\mu/\mu) \ dx\geq 0$ for all $\mu>0$.

\medskip
Since $f= f^+-f^-$, where $f^\pm= \max\{\pm f, 0\}$, it follows from \eqref{eq:deslog} that
$$
\begin{aligned}
   \Lambda_1\int_0^t\int_\Omega |\nabla \mathcal{H}u_\mu|^2 \ dxdt
   +\int_\Omega u_\mu(t) \log^{+}(u_\mu(t)&+\mu) \ dx
\\[5pt]
\leq \int_{\Omega} u_0 \log(u_0+\mu) \ dx&+\mu\int_{\Omega}  \log(1+u_0/\mu) \ dx
\\[5pt]
  &+\int_\Omega u_\mu(t) \log^-(u_\mu(t)+\mu)dx.
\end{aligned}
$$

Observe that the right hand side of the above inequality is bounded w.r.t. $\mu$ (small enough), because $u_\mu$ is bounded in $L^{\infty}(\Omega_T)$ w.r.t. $\mu$, 
and
$$
\int_\Omega u_\mu(t)\log^-(u_\mu(t)+\mu)dx
$$
is bounded w.r.t. $\mu$  (small enough).
Consequently, we have that
$\nabla \mathcal{H}u_{\mu}$ is (uniformly) bounded in $L^2(\Omega_T)$.

\medskip
On the other hand, using \eqref{normequi} and the Poincare inequality ( Corollary \ref{poincare} ) we deduce that 
\begin{equation}
\begin{aligned}
   \iint_{\Omega_T} \left|\nabla \mathcal{K}u_\mu(t,x)\right|^2dxdt&\leq \Lambda_1^{-1} \iint_{\Omega_T} \left|\mathcal{L}^{1/2-s}u_\mu(t,x)\right|^2dxdt
   \\[5pt]
& \leq\Lambda_1^{-1}\lambda_1^{-s}\iint_{\Omega_T} \vert\mathcal{L}^{1/2-s/2}u_\mu(t,x)\vert^2dxdt
\\[5pt]
&\leq \Lambda_1^{-1}\lambda_1^{-s}\Lambda_2\iint_{\Omega_T} \vert\nabla\mathcal{H}u_\mu(t,x)\vert^2dxdt.
\end{aligned}\nonumber
\end{equation}

Therefore, $\nabla \mathcal{K}u_\mu$ is (uniformly) bounded in $L^2(\Omega_T)$ w.r.t. $\mu$, and thus we obtain (along suitable subsequence) 
that $\nabla\mathcal{K}u_\mu$ converges weakly to $\textbf{v}$ in $\big(L^2(\Omega_T\big)^n$.
%i.e.
%$$
%\lim_{\mu \to 0}\int_0^T\int_{\Omega}\nabla\mathcal{K}u_{\mu}\cdot\Psi(x,t)\,dx\,dt=\int_0^T\int_{\Omega}\nabla\mathcal{K}u\cdot\Psi(x,t)\,dx\,dt
%$$
%where $\Psi(x,t)=(\psi_1, \psi_2, \cdots,\psi_n)$ and $\psi_i\in L^2(\Omega_T)$ for all $i=1,\cdots, n$.
It remains to show that $\textbf{v}=\nabla\mathcal{K}u$. 

Using the same ideas as in the proof of the proposition \ref{Thprinc}. Is is possible to select a subsequence, still denoted by $\{u_\mu\}$, converging weakly to $u$ in $L^2\left(0,T;D\big( \mathcal{L}^{(1-s)/2}\big)\right)$ such that $\textbf{v}=\nabla\mathcal{K}u$ in $\big(L^2(\Omega_T\big)^n$. Hence $\nabla\mathcal{K}u_\delta$ converges weakly to $\nabla\mathcal{K}u$  in $\big(L^2(\Omega_T\big)^n$.

\medskip
Now, we prove strong convergence for $\{u_\mu\}_{\mu>0}$ in $L^2(\Omega_T)$. To show that, we apply again the Aubin-Lions 
compactness Theorem. 
Since the coefficient $a_{i,j}$ of the matrix $A(x)$ is in $C^1(\hat{\Omega})$, together with the boundedness of $\nabla\mathcal{K}u_\mu$ in $L^2(\Omega_T)$, 
and the uniform limitation of $u_\mu$, we deduce from \eqref{estimateH-14.2} the following
we have
\begin{equation}
\int^T_0 \left\Vert \partial_t u_\mu \right\Vert^2_{H^{-1}(\Omega)} dt \leq C.
\label{eq:L2H-14.3}
\end{equation}
Passing to a subsequence (still denoted by $\{u_\mu\}$), we obtain that 
$$
  \text{$\partial_t u_\mu$ converges weakly to
 $\partial_t u$ in $L^2(0,T;H^{-1}(\Omega))$.}
$$
Applying the Aubin-Lions compactness Theorem, it follows that $u_\mu$ converges strongly to $u$ (along suitable sequence) in $L^2(\Omega_T)$.
Consequently, we obtain that $q(u_\mu)\nabla\mathcal{K}u_\mu$ converges weakly to $u \ \nabla\mathcal{K}u$ as $\mu\to 0^+$. Then, we are ready to pass to the limit in \eqref{eq:thequiaproxso13} as $\mu\to 0^+$ to get 
\begin{equation}
    \iint_{\Omega_T} u(t,x) \big( \partial_t\varphi(t,x) -  A(x)\nabla \mathcal{K}(u(t,x)) \cdot \nabla \varphi(t,x) \big)  \ dxdt
   +\int_{\Omega}u_0(x)\varphi(0,x)dx=0,\nonumber
\end{equation}
for all $\vp \in C^{\infty}_c((-\infty,T)\times\Bbb{R}^n)$. According to the Equivalence Theorem \ref{THEQUIVA}, 
we obtain the solvability of the IBVP \eqref{FTPME}.
\end{proof}

\begin{corollary}
The weak solution $u$ of the IBVP \eqref{FTPME} given by Theorem \ref{Thprincipal},
satisfies: 
\begin{enumerate}
\item[$(1)$]
For almost all $t\in (0,T)$, we have
\begin{equation}
\Vert u(t)\Vert_{\infty} \leq  \Vert u_0 \Vert_{\infty}, \quad \text{and} \label{eq:bounded4.3}
\end{equation}
\begin{equation}
\int_{\Omega}u(x,t)dx = \int_{\Omega}u_0(x) \ dx. \label{eq:consevationofmass4.3}
\end{equation}
Moreover, $0\leq u(t,x)$ a.e. in $(0,T)\times \Omega$.
\item[$(2)$] First energy estimate: For almost all  $t \in (0,T)$, 
\begin{equation}
    \Lambda_1\int^t_0\int_{\Omega}\vert \nabla \mathcal{H}u \vert^2 \ dxdt'
   +\int_{\Omega} u(t) \log(u(t)) \ dx \leq \int_{\Omega} u_0 \log(u_0) \ dx.\label{energyinequa4.3}
\end{equation}
\item[$(3)$] Second energy estimate: For almost all $0<t_1<t_2<T$,
\begin{equation}
\frac{1}{2}\int_{\Omega}\vert \mathcal{H} u(t_2) \vert^2 \ dx  + 
\Lambda_1\int_{t_1}^{t_2} \int_{\Omega} u\vert \nabla \mathcal{K}u\vert^2 \,dx\,dt \leq\frac{1}{2}\int_{\Omega}\vert \mathcal{H} u(t_1) \vert^2 dx.
\label{eq:2energystimate4.3}
\end{equation}
\end{enumerate}
\end{corollary}

\begin{proof} The proof of (\eqref{eq:bounded4.3} to \eqref{eq:2energystimate4.3}) is standard, see \cite{Hua}.
\end{proof}

%%%%%%%%%%%%%%%%%%%%%%%%
%%%%%%%%%%%%%%%%%%%%%%%%
 \section*{Acknowledgements}
%%%%%%%%%%%%%%%%%%%%%%%%
The author wishes to  thanks 
the Federal University of Alagoas, where the paper was
written, for the invitation and hospitality. This work was supported by CAPES.
%%%%%%%%%%%%%%%%%%%%%%

\end{document}